\documentclass[notitlepage,11pt,reqno]{amsart}

\usepackage{amsopn,esint}
\usepackage[final]{hyperref}
\usepackage[T1]{fontenc}
\usepackage[utf8]{inputenc}
\usepackage{apptools}
\AtAppendix{\counterwithin{lemma}{section}} 
\usepackage[margin=1in]{geometry}
\allowdisplaybreaks
\newcommand{\stkout}[1]{\ifmmode\text{\sout{\ensuremath{#1}}}\else\sout{#1}\fi}

 \newcommand{\grad}{\triangledown}

\newcommand{\sA}{\mathscr{A}}

\usepackage{graphicx,enumitem,dsfont,upgreek,nicefrac}
 \usepackage[dvips]{epsfig}
 \usepackage[mathscr]{eucal}
\usepackage{amscd}
\usepackage{amssymb}
\usepackage{amsthm}
\usepackage{amsmath}
\usepackage{latexsym}
\usepackage{dsfont}
\usepackage{upref}
\usepackage{hyperref}

\usepackage{color}
\theoremstyle{plain}

\newtheorem{thm}{Theorem}[section]
\theoremstyle{plain}
\newtheorem{lem}[thm]{Lemma}
\newtheorem{prop}[thm]{Proposition}

\theoremstyle{definition}
\newtheorem{defi}{Definition}[section]
\newtheorem{rem}{Remark}[section]
\newtheorem*{maintheorem*}{Main Theorem}
\newtheorem*{maincorollary*}{Main Corollary}
{%
\setcounter{enumi}{0}

\begin{enumerate}}%
{\end{enumerate} }

{%
\setcounter{enumi}{0}

\begin{enumerate}}%
{\end{enumerate} }

\newcommand{\norm}[1]{\ensuremath{\left\|#1\right\|}}
\newcommand{\abs}[1]{\ensuremath{\left|#1\right|}}

\newcommand{\Exp}{\mathbb{E}}
\newcommand{\Prob}{\mathbb{P}}

\newcommand{\sK}{\mathscr{K}}

\newcommand{\cG}{\ensuremath{\mathcal{G}}}

\newcommand{\cL}{\ensuremath{\mathcal{L}}}

\newcommand{\R}{\ensuremath{\mathbb{R}}}

\newcommand{\Usm}{\mathscr{U}}

\newcommand{\rd}{\ensuremath{\R^d}}

\newcommand{\dy}{\ensuremath{\, dy}}
\newcommand{\dz}{\ensuremath{\, dz}}
\newcommand{\ds}{\ensuremath{\, ds}}

\numberwithin{equation}{section} \allowdisplaybreaks

\DeclareMathOperator*{\Argmin}{arg\,min}
\usepackage{cancel,pdfsync}

\begin{document}
\title[Nonlocal ergodic control problems]{Nonlocal ergodic control problem in $\rd$}

\author{Anup Biswas}

\address{Indian Institute of Science Education and Research-Pune, Dr.\ Homi Bhabha Road, Pashan, Pune 411008. Email:
{\tt anup@iiserpune.ac.in}}

\author{Erwin Topp}

\address{
Instituto de Matemáticas, Universidade Federal do Rio de Janeiro, Rio de Janeiro - RJ, 21941-909, Brazil; and 
Departmento de Matem\'aticas Y Ciencias de la Computaci\'on, Universidad de Santiago de Chile,
Casilla 307, Santiago, Chile. Email: {\tt etopp@im.ufrj.br, erwin.topp@usach.cl}}

\begin{abstract}
We study the existence-uniqueness of solution $(u, \lambda)$ 
to the ergodic Hamilton-Jacobi equation
$$(-\Delta)^s u + H(x, \grad u) = f-\lambda\quad \text{in}\; \rd,$$
and $u\geq 0$, where $s\in (\frac{1}{2}, 1)$. We show that the 
critical $\lambda=\lambda^*$, defined as the infimum of all $\lambda$
attaining a non-negative supersolution, attains a nonnegative solution $u$. Under suitable conditions, it is also shown that $\lambda^*$ is
the supremum of all
$\lambda$ for which a non-positive subsolution is possible. Moreover,
uniqueness of the solution $u$, corresponding to $\lambda^*$, is also established. Furthermore, we provide a probabilistic
characterization that determines the uniqueness of the pair $(u, \lambda^*)$ in the class of all solution pair
$(u, \lambda)$ with $u\geq 0$.
Our proof technique involves both analytic and probabilistic methods in combination with a new local Lipschitz estimate obtained in this article.
\end{abstract}


\maketitle

\section{Introduction}

In this article, we are interested in the 
study of ergodic problems for viscous Hamilton-Jacobi equations in the Euclidean space $\R^d$ with fractional Laplacian. 
The model problem takes the form
\begin{equation}\label{eqmodel}
(-\Delta)^s u + \frac{1}{m}|\grad u|^m = f - \lambda \quad \mbox{in} \ \rd,
\end{equation}
where $s \in (\nicefrac{1}{2},1)$, $m > 1$, $\lambda \in \R$, $f \in C(\rd)$ and $\grad u$ denotes the gradient of $u$. For an adequately regular function $u: \rd \to \R$, $(-\Delta)^s u$ denotes the fractional Laplacian of order $2s$, which is explicitly defined as
$$
(-\Delta)^s u(x) = -c_{d,s} \mathrm{P.V.} \int_{\rd} (u(x + y) - u(x))|y|^{-(d + 2s)} dy,
$$
where $\mathrm{P.V.}$ stands for the Cauchy principal value, and 
$
c_{d,s} = \frac{4^s\Gamma(d/2+s)}{\pi^{d/2}|\Gamma(-s)|}
$
is a well-known normalizing constant making $(-\Delta)^s \to -\Delta$ in an adequate functional framework, see~\cite{Hitch}.

We provide existence of an eigenpair $(u, \lambda)$ solving the problem~\eqref{eqmodel}, and qualitative properties of the solutions in the context of viscosity solution's theory. The main assumption here is the coercivity of the source term $f$, which allows us to deal with the lack of compactness of the state variable in~\eqref{eqmodel}.

In the periodic setting (namely, when $f$ is $\mathbb Z^d$-periodic), problem~\eqref{eqmodel} have been addressed in~\cite{BCCI14, BKLT15, BLT-17}, involving more general nonlocal operators $I$ with the form
\begin{equation}\label{opgeral}
-Iu(x) = -\mathrm{P.V.} \int_{\rd} (u(x + z) - u(x))K(z)dz,
\end{equation}
where $K: \rd \setminus \{ 0 \} \to \R_+$ is a symmetric function (the kernel) satisfying nowadays standard integrability conditions at the origin, and a growth conditions at infinity. The fractional Laplacian corresponds to the particular case $K(z) = c_{d,s}|z|^{-(s + 2s)}$.

Periodicity induces robust compactness properties to the equation, which in addition to a priori regularity estimates and strong maximum principle lead to uniqueness for the eigenvalue $\lambda$, and uniqueness, up to an additive constant, for the eigenfunction $u$ (in the sequel, we will refer to this property as uniqueness of the ergodic problem). The method used there follow the classical arguments of Lions, Papanicopoau \& Varadhan~\cite{LPV}, and do not differ substantially among the case of dominant diffusion ($2s \geq m$) and/or the gradient dominant case ($2s < m$). In \cite{N19}, the author addresses the problem in the non-periodic setting for nonlocal Ornstein-Uhlenbeck type equations
with an exponentially decaying nonlocal kernel.
Compactness is a consequence of the confining influence of the drift term and the exponential decay of the nonlocal kernel. This allows to get uniqueness of the eigenpair as in the periodic case.

Similar ergodic problems have been extensively studied for the second-order analogue of~\eqref{eqmodel}, namely the problem
\begin{equation}\label{viscous}
-\Delta u + \frac{1}{m}|\grad u|^m = f - \lambda \quad \mbox{in} \ \rd.
\end{equation}

The existence and uniqueness for the ergodic problem~\eqref{viscous} in the periodic case is provided in Barles \& Souganidis~\cite{BarSou01}. In the non-periodic case, the very first work treating the case of quadratic Hamiltonian
was done by Bensoussan \& Freshe \cite{BenFre92}. In both references, the results are obtained using analytic methods. Ichihara~\cite{Ichi2011,Ichi2015} uses probabilistic methods and the optimal control perspective of the problem to obtain the uniqueness of the ergodic problem in the class of nonnegative solutions. Some other important works on this topic include Cirant \cite{Cirant}, Barles \& Meireles \cite{BM16}, Arapostathis, Biswas \& Caffarelli \cite{ABC19}, 
Arapostathis, Biswas \& Roychowshury \cite{ABR22}. It is also worth to mention the application of stationary ergodic problems to the study of large time behavior for solutions of associated parabolic equations, see for instance~\cite{BarSou01, Fujita-06a, Tchamba-10, Ichihara-12, Barles-10, BQR21} in the local case, and the mentioned works~\cite{BCCI14, BKLT15, BLT-17} in the nonlocal setting.

Other than the periodic setting, nonlocal ergodic problems in $\rd$ have been considered in quite restricted settings. In addition to the mentioned article~\cite{N19}, Br\"{a}ndle \& Chasseigne \cite{BC19} consider a problem similar to ours, but for the nonlocal operator
of dispersal type (that is, $K$ in~\eqref{opgeral} is continuous, bounded and compactly supported). In particular, none of these references deal with the relevant case of the fractional Laplacian.

Now we present the assumptions that allow us to get our main result. For introductory purposes we present it in the context of problem~\eqref{eqmodel}, and later the actual results are presented
in a greater generality.

As we mentioned above, our results rely on the coercivity of the source term $f \in C(\rd)$, namely
\begin{itemize}
\item[\hypertarget{F1}(F1)] There exists $C > 0$ and $\gamma < m(2s-1)$ such that $|f(x)|\leq C(1+|x|^\gamma)$ for all $x\in\rd$.

\medskip

\item[\hypertarget{F2}(F2)] $f$ is coercive, that is, 
$$\lim_{|x|\to\infty} f(x)= +\infty.$$
\end{itemize}

As we will see later, coercivity condition~\hyperlink{F2}{(F2)} helpful to obtain unbounded solutions. As well known, the solutions must be restricted to a class ad-hoc to the $s$-fractional operator. For this, we denote $\omega_s(y)=(1+|y|^{d+2s})^{-1}$ and look for viscosity solutions $u$ belonging to the space $L^1(\omega_s)$, defined as
$$
L^1(\omega_s) = \{ u \in L^1_{loc}(\rd) : \int_{\rd} |u(y)| \omega_s(y)dy < \infty \}.
$$

We also denote by $C_+(\rd)$ (resp. $C_-(\rd)$) the set of nonnegative (resp. nonpositive) functions in $C(\rd)$.
Our main result is the following
\begin{thm}\label{TeoIntro}
Let $s \in (1/2, 1)$, $m > 1$, and let $f \in C(\rd)$ satisfying~\hyperlink{F1}{(F1)}--\hyperlink{F2}{(F2)}. Then, there exists  $(u, \lambda^*) \in C(\rd) \cap L^1(\omega_s) \times \R$ solving the ergodic problem
\begin{equation}\label{eqTeoIntro}
 (-\Delta)^s u + \frac{1}{m}|\grad u|^m = f - \lambda^* \quad \text{in}\; \rd.
\end{equation} 

The ergodic constant $\lambda^*$ meets the characterization formulae
\begin{align*}
\lambda^* &=  \inf\{\lambda\; : \; \exists \, u\in C_+(\rd)\cap L^1(\omega_s)\; 
\text{such that}\; (-\Delta)^s u + \frac{1}{m}|\grad u|^m \geq f-\lambda\; \text{in}\; \rd\}, 
\\
 & =  \sup\{\lambda\; : \; \exists \, u\in C_-(\rd)\cap L^1(\omega_s)\; 
\text{such that}\; (-\Delta)^s u + \frac{1}{m}|\grad u|^m \leq f-\lambda\; \text{in}\; \rd \}.
\end{align*}
The function $u$ is unique up to an additive constant. If $f$ is locally H\"older continuous, then $u \in C^{2s+}(\rd)$.

Moreover, if $f_1, f_2$ satisfy~~\hyperlink{F1}{(F1)}--\hyperlink{F2}{(F2)} , and if $\lambda^*(f_i)$ denotes
 the corresponding critical value in \eqref{eqmodel} associated to $f_i$, $i=1,2$,
then for $f_1\lneq f_2$ we have $\lambda^*(f_1)<\lambda^*(f_2)$.
\end{thm}

As we mentioned above, Theorem~\ref{TeoIntro} is a particular case of a more general result that we state in full generality in Theorem~\ref{T1.2} below. Here $C^{2s+}(\rd)$ denotes the subset of continuous functions with the following property: $g\in C^{2s+}(\rd)$ if and only if for each compact set $\sK$ there exists an $\eta>0$ satisfying $g\in C^{2s+\eta}(\sK)$. 

Before we proceed, we would like to  point out
some key features of the above results concerning the assumptions on $f$. Condition \hyperlink{F2}{(F2)} is pretty standard in the context of an ergodic
control problem, whereas \hyperlink{F1}{(F1)} will be used to construct an
appropriate supersolution which will serve as a barrier function. To some extent, condition \hyperlink{F1}{(F1)} is optimal for the existence of the solution, see Theorem~\ref{T-nonex} below.
More precisely, we show that if the growth of $f$ is bounded from below by $|x|^{m(2s-1)}$ then $\lambda^*=\infty$.

In the local setting, one normally assumes $f$ to be locally Lipschitz \cite{ABC19,BM16,Ichi2011}.
One of the main reasons for this assumption is the use of standard Bernstein estimate, which requires  $f$ to be Lipschitz. On the other hand, our Lipschitz estimate (see Theorem~\ref{Thm-Lip} below) works for continuous $f$. 
Furthermore, for $s=1$ and $m\in (1, 2)$ (that is, the subcritical case), the uniqueness of solutions are proved imposing some restrictions on $f$ (see \cite[Condition~(A1)]{ABC19}, \cite[Section~3.3]{BM16}). Our results do not impose any such restriction on $f$, and therefore, can be used to improve the results for the local setting. In~\cite{BC19}, the authors establish uniqueness of the solution (for their dispersal type nonlocal kernel) under some growth assumption on the solution and also imposing a set of conditions on $f$. In particular, it does not include $f$ that increases {\it slowly} to infinity (for example, $\log x$). Also, unlike \cite{BC19}, we can not treat the nonlocal term as a lower order (or zeroth order)
term.  For dispersal type kernels, non-existence of radially
increasing solutions for 
a suitable class of $f$ is obtained by \cite{BC19}. Since the 
nonlocal kernel did not have any singularity, the authors of
\cite{BC19} could transform the problem to an ode and analyze it to
establish non-existence. Our proof of Theorem~\ref{T-nonex} relies on comparison principle, and we do not require the supersolution to be radial.

It is important to point out that, unlike the results for second-order problems presented in~\cite{BM16}, we are not able to provide a full uniqueness result for the ergodic constant. In \cite{BM16}, this is established as a consequence of the comparison principle among equations with different \textsl{eigenvalues} and the use of Cole-Hopf transformation (see Proposition 3.2 in~\cite{BM16}).
In fact, this idea of exponential transformation goes back to \cite{BenFre92}. We are not able to reproduce such a technique in the nonlocal setting. We neither are able to provide examples where the eigenvalue is not unique.

\medskip

In the remaining of this introduction, we present our second main result. This is motivated by the relation of ~\eqref{eqmodel} to the study of stochastic ergodic control problems  where the external noise of the controlled dynamics is 
given by a $2s$-stable process. More precisely, we consider the controlled stochastic differential equation 
$$
d X_t= -\zeta_t dt +  dL_t \quad \text{for}\; t> 0,
$$
where $L$ denotes a $2s$-stable process and $\zeta$ is an {\it admissible} control taking values in $\rd$.  Given $x \in \rd$, we consider the minimization problem of long-run average cost
$$
\inf_{\zeta(\cdot)} \limsup_{T\to\infty}\, \frac{1}{T} \Exp_x \left[\int_0^T \left( \frac{1}{m'} |\zeta_t|^{m'} + f(X_t)\right) dt\right],
$$
where $m'$ is the H\"older conjugate of $m$, 
and the infimum is taken over all admissible control.  Under appropriate assumptions, it is possible to prove that this limit exists, independent of $x$, and coincides with $\lambda^*$ in Theorem~\ref{TeoIntro} (see, for instance, the proof of Lemma-3 in the supplement of \cite{ABB18}).


Assume that $f$ is locally H\"{o}lder continuous.
Consider a nonnegative, classical solution $u$ to~\eqref{eqTeoIntro}. Notice that in this case, the vector field  $x \mapsto b_u(x) :=  |\grad u(x)|^{m - 2} \grad u(x)$ is locally H\"{o}lder continuous. 
For $x, \xi \in \rd$, let us define
$$
\cG(x, \xi)=f(x)+ \frac{1}{m'} |\xi|^{m'},
$$
and the operator
$\sA_u$ as follows
$$
\sA_u \varphi(x)=-(-\Delta)^s \varphi(x)  - b_u(x)\cdot \grad \varphi(x),\quad \varphi\in C^{2s+}(\rd)\cap L^1(\omega_s).
$$
It is then easily seen from~\eqref{eqTeoIntro} that
\begin{equation*}
\sA_u u(x) + \cG(x, b_u(x)) 
= -(-\Delta)^s u + \min_{\xi\in\rd}[-\xi\cdot \grad u + \cG(x, \xi)]
= -(-\Delta)^s u - \frac{1}{m}|Du|^m + f = \lambda^*.
\end{equation*}
Note that the above equation corresponds to a stochastic 
ergodic control problem with the running cost function
given by $\cG$ and the controlled dynamics is given by
$$X_t = -\int_0^t \zeta_s ds + L_t,\quad t\geq 0,$$
where $\zeta$ denotes the control process and $L$ is a $2s$-stable process in $\rd$. For our next result, we are going to exploit 
this underlying control problem. We restrict ourselves to 
stationary Markov controls.

More precisely, by a stationary Markov 
control, we mean a locally H\"{o}lder continuous function $\zeta:\rd\to \rd$. 
The associated operator $\sA_\zeta$ is defined as
$$
\sA_\zeta \varphi(x)=-(-\Delta)^s \varphi(x) - \zeta(x)\cdot \grad \varphi(x),\quad \varphi\in C^{2s+}(\rd)\cap L^1(\omega_s).
$$
The solution $X$, corresponding to $\zeta$, would be understood
in the sense of martingale problem which we introduce below. Readers may consult Ethier and Kurtz \cite{EK86} for more details on this topic.

Let $\Omega:=\mathbb{D}([0, \infty):\rd)$ be the space of all right
continuous $\rd$ valued functions on $[0, \infty)$ with finite left limit, and
let $X:\Omega\to \rd $ denote the canonical coordinate process, that is,
$$X_t(\omega)=\omega(t)\quad \text{for all}\; \omega\in \Omega.$$

By $\{\mathfrak{F}_t\}$ we denote the filtration of $\sigma$-algebras generated by $X$. 
\begin{defi}\label{D1.1}
Let $\zeta$ be a stationary Markovian control and $x\in\rd$.
A Borel probability measure $\Prob$ on $\mathbb{D}([0, \infty):\rd)$ is said to be a solution to the martingale problem
for $(\sA_\zeta, C^\infty_c(\rd))$ with initial point $x$ if
\begin{itemize}
\item[(i)] $\Prob(X_0=x)=1$,
\item[(ii)] for every $\psi\in C^\infty_c(\rd)$, we have $\int_0^t |\sA_\zeta \psi(X_s)|ds<\infty$ $\Prob$-almost surely, for all $t>0$, and 
$$M^\psi_t:= \psi(X_t)-\psi(X_0) -\int_0^t \sA_\zeta \psi(X_s)ds$$
is $(\mathfrak{F}_t, \Prob)$-martingale.
\end{itemize}
We recall that a real-valued stochastic process $\{Y_t\; :\; t\geq 0\}$ is said to be
a martingale with respect to $(\mathfrak{F}_t, \Prob)$ if the following three properties hold
\begin{itemize}
\item[-] $Y_t$ is $\mathfrak{F}_t$ measurable for all $t\geq 0$,
\item[-] for each $t\geq 0$, we have $\Exp[|Y_t|]<\infty$ where $\Exp$ denotes the expectation with respect to $\Prob$,
\item[-] for each $t\geq s\geq 0$, we have $\Exp[Y_t|\mathfrak{F}_s]=Y_s$ $\Prob$-almost surely.
\end{itemize}
The martingale problem for $(\sA_\zeta, C^\infty_c(\rd))$ with initial point $x \in \rd$ is said to be well-posed if
it has a unique solution, { that is, there exists only one probability measure $\Prob$ on $\mathbb{D}([0, \infty):\rd)$ satisfying properties (i) and (ii) in Definition~\ref{D1.1}.}
\end{defi}
Given a set $D$, by $\breve\uptau_D$ we denote the first return time to 
$D$, that is,
$$\breve\uptau_D=\inf\{t>0\; :\; X_t\notin D^c\}.$$

\begin{defi}
Let $\Usm$ be the set of all stationary Markov controls $\zeta$ satisfying the following:
\begin{itemize}
\item[(i)] The martingale problem for $(\sA_\zeta, C^\infty_c(\rd))$ with initial point $x\in\rd$ is well-posed for all $x\in\rd$.
\item[(ii)] For any non-empty compact set $\mathscr{B}$,
containing $\{x\in \rd\; :\; \cG(x, \xi)-\lambda^*\leq 1\; \text{for some}\; \xi\}$, we have $\Prob^\zeta_x(\breve\uptau_{\mathscr{B}}<\infty)=1$ for all $x\in {\mathscr{B}}^c$, 
where $\Prob^{\zeta}_x$ denotes the unique probability measure solving the martingale problem for $(\sA_\zeta, C^\infty_c(\rd))$ with initial point $x$.
\end{itemize}
\end{defi}

{ Condition (ii) above is nothing but the recurrence property of $X$ with respect to compact sets $\mathscr{B}$ containing the bounded set $\{x\in \rd\; :\; \cG(x, \xi)-\lambda^*\leq 1\; \text{for some}\; \xi\}$. In other words, for each $\zeta\in\Usm$ we need the stochastic process $X$ to enter $\mathscr{B}$ with probability $1$ with respect to $\Prob^{\zeta}_x$. This property
is required to control the last term in the display \eqref{ineqexit} which is crucial for our uniqueness result in the next theorem. In fact, using the non-degenerate property of the fractional Laplacian and the strong Markov property of $X$ (which follows from (i) above) it is possible to show that $X$ is recurrent with respect to any non-empty open set if it is recurrent with
respect to some non-empty compact set $\mathscr{B}$ (see for instance, \cite[Theorem~5.1]{ABC16}).
Moreover, by Lemma~\ref{L6.4}, which appears in Section~\ref{secstoch}, we see that $\Usm$ is non-empty. }

Now we can state our next result.
\begin{thm}\label{TeoIntro2}
Assume the hypotheses of Theorem~\ref{TeoIntro} holds, and $f$ is locally H\"{o}lder continuous. Let $u$ be a solution given in Theorem~\ref{TeoIntro}. 

Then, for any $\zeta\in \Usm$ and compact set
$\mathscr{B}\supset \{x\in \rd\; :\; \cG(x, \xi)-\lambda^*\leq 1\; \text{for some}\; \xi\}$, we have
\begin{equation}\label{ineqexit}
u(x)\leq \Exp^\zeta_x\left[\int_0^{\breve\uptau_{\mathscr{B}}}(\cG(X_s, \zeta(X_s))-\lambda^*)\ds\right] 
+ \Exp_x^\zeta[u(X_{\breve\uptau_{\mathscr{B}}})]\quad \text{for}\; x\in \mathscr{B}^c,
\end{equation}
where $\Exp^\zeta_x[\cdot]$ denotes the expectation operator with respect to the probability measure $\Prob^{\zeta}_x$.

On the other hand, if there exists a solution $(v, \lambda) \in C^{2s+}(\rd) \cap L^1(\omega_s) \times \R, v\geq 0$, to~\eqref{eqmodel}
and $(v,\lambda)$ satisfies~\eqref{ineqexit}, then $\lambda=\lambda^*$ and $v=u +c$ for some constant $c$.
\end{thm}

As before, Theorem~\ref{TeoIntro2} is a particular case of Theorem~\ref{T1.4} below.
Formula~\eqref{ineqexit} is often referred to as a stochastic representation of the solution. Note that Theorem~\ref{TeoIntro2} provides uniqueness of the ergodic problem~\eqref{eqmodel} in the class of all nonnegative solutions $u$ for which~\eqref{ineqexit} holds. Such a requirement is not necessary for the second-order problems related to~\eqref{eqmodel}. As we already mentioned, in~\cite{BenFre92,BM16} uniqueness is obtained through Cole-Hopf transformation, whereas in~\cite{ABC19,ABR22} this is a consequence of an occupation measure based approach together with a {\it strong} Bernstein estimate.
Both approaches do not seem to be helpful in treating the nonlocal ergodic control problem in hand. 

\medskip

\subsection*{Basic notations and organization of the paper.} We start this section with some general notation. From here and until the end of this article, for $r > 0$ and $x \in \R^d$ we denote by $B_r(x)$ the ball of radius $r$ and center $x$. When $x=0$ we just write $B_r = B_r(0)$, and if in addition $r=1$ we write $B = B_1(0)$. For $a \in \R$ we adopt the notation $a^+ = \max \{ a, 0 \}$ and $a^- = \min \{ a, 0 \}$, from which $a = a^+ + a^-$. Given a set $E \subseteq \rd$ and a bounded function, $u: E \to \R$ we denote $\mathrm{osc}_E u = \sup_E u - \inf_E u$.

\medskip

The paper is organized as follows: In Section~\ref{secLip} we provide more general assumptions on the nonlocal term and the Hamiltonian, and provide the local Lipschitz estimates for the solutions. In Section~\ref{secgeneral} we  state our first main result Theorem~\ref{TeoIntro} in the general form and provide the proof of the existence of the eigenpair. In Section~\ref{seccharact} we prove the characterization of the critical eigenvalue, and in Section~\ref{secunique} we show the uniqueness up to a constant for the associated eigenfunctions. Finally, in Section~\ref{secstoch} we prove our second main Theorem~\ref{TeoIntro2} related to the stochastic characterization of the critical eigenvalue in a larger generality.


\section{A priori Local Lipschitz estimates}
\label{secLip}

In what follows, we introduce our more general problem, starting with the Hamiltonian $H \in C(\rd \times \rd)$.
\begin{itemize}
\item[\hypertarget{H1}{(H1)}] There exists $m \geq 1$ such that, for each $R>0$ there exists $C_{H,R}>0$ so that
\begin{equation*}
|H(x, p+q)-H(y, p)|\leq C_{H,R}\left[ |x-y| (1+|p|^m) + |q| (|p|^{m-1}+|q|^{m-1})\right]
\end{equation*}
for all $ x, y\in B_R$ and $p, q\in\rd$ with $|q| \leq R$.

\item[\hypertarget{H2}{(H2)}] There exists a constant 
$\upkappa$ such that
$$\frac{1}{\upkappa}|p|^m-\upkappa \leq H(x, p) 
\quad \text{for all}\; x, p\in\rd.$$
\end{itemize}

The parameter $m$ encodes the gradient growth of the Hamiltonian. A standard example of $H$ satisfying \hyperlink{H1}{(H1)}, \hyperlink{H2}{(H2)} is given by
\begin{equation*}
H(x, p)= \frac{1}{m}\langle p, a(x) p\rangle^{\nicefrac{m}{2}} + b(x)\cdot p\quad x, p\in\rd,
\end{equation*}
where $a:\rd\to\R^{d\times d}, b:\rd\to \rd$ are Lipschitz continuous functions and $a$ is uniformly positive definite. Note that unlike some of the existing works \cite{ABC19,ABR22,BM16,Ichi2011} we do not assume any convexity or 
$m$-homogeneity in the $p$ variable.

We also consider a slightly more general nonlocal operator, and for this we introduce some extra notation. For a kernel $K: \R^d \setminus \{ 0 \} \to \R$ satisfying the ellipticity condition
\begin{equation}\label{ellipticity}
\lambda |y|^{-(d + 2s)} \leq K(y) \leq \Lambda |y|^{-(d + 2s)} \quad \mbox{for} \ y \in \R^d \setminus \{ 0 \},
\end{equation}
for some constants $0 < \lambda \leq \Lambda < +\infty$, 
and given a suitable measurable function $u: \R^d \to \R$, we denote
\begin{equation}\label{gen-I}
I u(x) = \int_{\rd} (u(x+y)-u(x)-\chi_{B}(y)\, y\cdot \grad u(x))K(y)dy,
\end{equation}
where $\chi_E$ denotes the indicator function of the (measurable) set $E \subset \R^d$. Notice that no simmetry assumption is imposed on $K$.

We also require the following notation for the viscosity evaluation: given $E, u$ as before, and $p \in \rd$, we denote
\begin{equation}\label{Ieval}
I[E](u, p, x) = \int_E (u(x+y)-u(x)-\chi_{B}(y)\, y\cdot p )K(y)dy.
\end{equation}

For simplicity, whenever $u$ is smooth at $x$, we write
\begin{equation*}
I[E](u, x) = \int_E (u(x+y)-u(x)-\chi_{B}(y)\, y\cdot \grad u(x))K(y)dy.
\end{equation*}

Finally, given $u \in L^1(\omega_s)$, $E \subset \rd$ measurable, and $a > 0$, we denote
$$
\mathscr{H}(u, E, a):= \sup_{x\in E} \int_{\rd} \frac{|u(x+y)-u(x)|}{a + |y|^{d + 2s}} \dy.
$$

Notice that $a \mapsto \mathscr{H}(u,E,a)$ is decreasing, and if $a < 1$, then $\mathscr{H}(u,E,a) \leq a^{-1}\mathscr{H}(u,E,1)$. For $a=1$, we simply write $\mathscr{H}(u,E) = \mathscr{H}(u,E,1)$.

The main result of this section is the following
\begin{thm}\label{Thm-Lip}
Let $K$ satisfying~\eqref{ellipticity}, $\alpha \geq 0$, $f \in C(\rd)$ and $H$ satisfying~\hyperlink{H1}{(H1)}. Assume further~\hyperlink{H2}{(H2)} holds when $m \geq 2s + 1$.
For $R>0$, let $u \in C(\R^d) \cap L^1(\omega_s)$ be a viscosity solution to the problem
\begin{equation}\label{Lip-01}
\alpha u - I u + H(x,\grad u) - f(x) = 0 \quad \mbox{in} \ B_{R + 2}.
\end{equation}

Then, there exists a constant $C_R > 0$ such that
\begin{equation*}
|u(x) - u(y)| \leq C_R |x - y| \quad \mbox{for all} \ x, y \in B_R.
\end{equation*}

The constant $C_R$ can be estimated as follows: let $C_{H, R}$ as in (H1), and 
consider $\mathfrak{Z}_1, \mathfrak{Z}_2$ given in~\eqref{Zeta1}, ~\eqref{Zeta2} below. Then, 
for each $\theta_0 \in (0,1)$ small enough,  there exists $C > 0$, just depending on the data,
but not on $R$ nor $u$ such that $C_R$ above has the form
\begin{equation*}
C_R = C C_{H, R+2}^{\frac{1}{1 - \theta_0}} \max \{ \mathfrak{Z}_1, \mathfrak{Z}_2 \}^{\frac{m - \theta_0}{1 - \theta_0}}
\end{equation*}
when $m \geq 2s + 1$, and
\begin{equation}\label{estBernstein}
C_R = C C_{H, R+2}^{\frac{1}{2s + 1 - m - \theta_0}} \mathfrak{Z}_2^{\frac{2s - \theta_0}{2s + 1 - m - \theta_0}}
\end{equation}
when $m < 2s + 1$.
\end{thm}

Before providing the proof of this result, let us briefly comment some previous related contributions about regularity. We firstly mention that for $s < 1/2$ and $m > 2s$, Lipschitz regularity follows from the classical arguments for eikonal-type problems, as can be seen in~\cite{BKLT15, BT16N}. 

A different approach is required in the case $s \geq 1/2$. Using Ishii-Lions method, Barles, Chasseigne, Ciomaga \& Imbert~\cite{BCCI12} prove Lipschitz regularity for equations with dominating diffusion. The regularizing effect of the ellipticity requires two main ingredients: the diffusion is of order $2s > 1$, and appropriate regularity on the data to control the effect of the gradient. The borderline case $s=1/2$ is addressed in~\cite{CGT-22} with the same arguments. In~\cite{BLT-17}, the authors get Lipschitz estimates for solutions of degenerate elliptic equations with Lipschitz data, following the (viscosity) weak Bernstein technique introduced in~\cite{Barles-91}. Here there is no restriction on the order of the nonlocality, and the regularity basically comes from the superlinear coercivity of the gradient.  Other results for the case $s > 1/2$ are available for equations with (at most) linear growth on the gradient, see for instance~\cite{BK22,CDV22}. 

Our approach is basically the one presented in~\cite{BCCI12}, but with appropriate modifications to deal with the unbounded solutions. The (Lipschitz) continuity of the data moderates the effect of the gradient term in the doubling variables procedure. Roughly speaking, it ``subtracts $1$ to the exponent" of the power-type growth of the gradient. This explains why we do not require the extra coercivity assumption \hyperlink{H2}{(H2)} when $m < 2s + 1$.  We stress on the fact that assumption (H1) could be weakened (say, assuming H\"older continuity on $x$ just as in~\cite{BCCI12}), but we prefer to keep the Lipschitz assumption on $H$ for simplicity. The more difficult scenario is when $m>2s$ (in fact, when $m \geq 2s + 1$), where we combine Ishii-Lions method with apriori  H\"{o}lder estimates from \cite{BKLT15}. We recall this for a sake of completeness.

\begin{prop}\label{P2.1}
Assume hypotheses of Theorem~\ref{Thm-Lip} hold with $m \geq 2s + 1$, and 
denote $\gamma_0=\frac{m-2s}{m-1}$. Then, there exists a constant $C > 0$ such that, for all $R > 1$ and each viscosity subsolution $u \in L^1(\omega_s)$ to~\eqref{Lip-01}, we have
$$
\sup_{B_{R+\frac{5}{4}}}\frac{|u(x)-u(y)|}{|x-y|^{\gamma_0}}\leq C R^{1 - \gamma_0} (A^{1/m} + {\rm osc}_{B_{R+2}} u),
$$
where 
\begin{equation}\label{defA}
A = A_R =
1 + 
\sup_{B_{R+2}}f^+ - \alpha \inf_{B_{R + 2}}u^- + \mathscr{H}(u, B_{R+2}). 
\end{equation}
\end{prop}

\begin{proof}
Set $\bar{u}=u-\inf_{B_{R+2}} u$, and notice that $\bar u$ is nonnegative in $B_{R + 2}$ and that $\mathrm{osc}_{B_{R + 2}} \bar u = \mathrm{osc}_{B_{R + 2}} u$. It is direct to see, in view of \hyperlink{H2}{(H2)} and the invariance under additive constants of $I$, that $\bar u$ solves the inequality
$$
-I \bar u + \frac{1}{\upkappa} |\grad \bar u|^m \leq \upkappa + f - \alpha \inf_{B_{R + 2}} u^- \quad \mbox{in} \ B_{R + 2},
$$
in the viscosity sense.

Next, consider a smooth cut-off function $\eta:\rd\to [0, 1]$ satisfying
\[
\eta(x)=\left\{
\begin{array}{ll}
   1  & \text{for} \; x\in B_{R+\frac{7}{4}}, \\[2mm]
0  & \text{for} \; x\in B^c_{R+\frac{15}{8}}.
\end{array}
\right.
\]

Define $v=\eta \bar{u}$. We have that $\mathrm{osc}_{B_{R + 2}} v \leq \mathrm{osc}_{B_{R + 2}} u$. By linearity we have $I \bar u = I v + I((1 - \eta)\bar u)$, and therefore we can write
$$
-I v + \frac{1}{\upkappa }|\grad v|^m \leq \upkappa + \sup_{B_{R+2}} f^+ - \alpha \inf_{B_{R+2}}u^- + I((1-\eta)\bar{u})
\quad \text{in}\; B_{R+\frac{7}{4}}.
$$

Since $\eta(x+y)=1$ for $x\in B_{R+\frac{13}{8}}$ and $|y|\leq \nicefrac{1}{8}$, it follows that for each $x \in B_{R + \frac{13}{8}}$ we have
\begin{align*}
I((1-\eta)\bar{u})(x) = &
\int_{|y| \geq \frac{1}{8}} (1 - \eta(x + y)) \bar u(x + y) K(y)dy \\
\leq &  C_1 \Lambda \left( {\rm osc}_{B_{R+2}} u + \sup_{x\in B_{R+2}}\int_{|y|\geq \frac{1}{8}} \frac{|u(x+y)-u(x)|}{|y|^{d+2s}} dy\right) \\
\leq & C_1 \Lambda \left( {\rm osc}_{B_{R+2}} u + 2 * 8^{d + 2s} \mathscr{H}(u, B_{R + 2}) \right),
\end{align*}
for some $C_1 > 0$ just depending on $d,s$. Thus, for some constant $C_2 >0$ just depending on $d,s, \upkappa$ and $\Lambda$ we have
$$
-I v + \frac{1}{\upkappa }|\grad v|^m\leq C_2 \left(1 + {\rm osc}_{B_{R+2}} u + 
\sup_{B_{R+2}}f^+ - \alpha \inf_{B_{R + 2}}u^- + \mathscr{H}(u, B_{R+2}) \right) \quad \mbox{in} \ B_{R + \frac{7}{4}},
$$

Thus, denoting $\tilde A$ the expression in brackets, we can use the interior H\"older estimates of \cite[Theorem~2.2]{BKLT15} to conclude that 
$$
\frac{|v(x) - v(y)|}{|x-y|^{\gamma_0}} \leq C(\tilde A^{1/m} + \mathrm{osc}_{B_{R + 2}} v) \quad \mbox{for all} \ x, y \in B_R, \quad |x - y| \leq \rho,
$$
for some $C, \rho > 0$ not depending on $R$.

Using that $v = \bar u$ in $B_R$ and since $m > 2s > 1$ we arrive at
\begin{equation*}
\frac{|u(x) - u(y)|}{|x-y|^{\gamma_0}} \leq C (A^{1/m} + \mathrm{osc}_{B_{R + 2}} u) \quad \mbox{for all} \ x,y \in B_R, \ |x - y| \leq \rho,
\end{equation*}
and using this estimate, a simple analysis yields to the estimate of $[u]_{C^{\gamma_0}(B_R)}$.
\end{proof}

\medskip

Now we are in a position to provide the

\medskip

\noindent
\textit{Proof of Theorem~\ref{Thm-Lip}:}
For $\theta \in (0,1)$ to be fixed later, consider the function $\varphi: [0,+\infty) \to [0,+\infty)$ defined as 
$
\varphi(t) = (t - t^{1 + \theta}/2)
$
for $t \in [0,1/2]$, and $\varphi(t) = \varphi(1/2)$ for $t > 1/2$.
Let $\psi : \R^d \to \R_+$ be a smooth, bounded function, radially nondecreasing and such that $\psi(x) = 0$ for $|x| \leq R$, and $\psi(x) = 1$ for $|x| \geq R+1/2$. In particular, $\psi$ has bounded first and second derivatives, independent of $R$. 

For simplicity, denote
$$
M_u=
{\rm osc}_{B_{R + 2}} u.
$$

For $L> 0$, now define
\begin{equation}\label{penalization}
\Phi(x,y) = \Phi_{L}(x,y) := u(x) - u(y) - L \varphi(|x - y|) - M_u\, \psi(x),
\end{equation}
and let 
\begin{equation}\label{defM}
    M := \sup_{x, y \in \bar B_{R + 2}} \Phi(x,y).
\end{equation}

We remark that the supremum is attained by continuity. 
We claim that there exists $L  > 1$ such that $M \leq 0$. This implies the desired local Lipschitz regularity for the solution.

Suppose, on the contrary, that for all $L$ large enough (to be given later on),
there exists a solution $u$ to \eqref{Lip-01} so that $M>0$.
Let $(\bar{x}, \bar{y}) \in \bar{B}_{R+2} \times \bar{B}_{R+2}$ be point attaining the maximum in~\eqref{defM}.
Since $M>0$, we have $\bar x \neq \bar y$ for all $L$, and from the definition of $\psi$ it follows that $|\bar{x}|\leq R+1/2$. Furthermore, if we let $L$ large enough to satisfy
\begin{equation}\label{firstchoice}
L\varphi(4^{-1}) > M_u,
\end{equation}
then we also have $\abs{\bar{x}-\bar{y}}\leq 1/4$, implying $|\bar{y}|\leq R+3/4$.

Since $\Phi(\bar{x},\bar{x})\leq \Phi(\bar{x},\bar{y})$, we also get
\begin{equation*}
L \varphi(|\bar x - \bar y|)\leq u(\bar x) - u(\bar y)  \leq M_u.
\end{equation*}
for all $L$. In particular, if we set $L$ to satisfy \eqref{firstchoice}, we must have $\varphi(|\bar{x}-\bar{y}|)\geq \frac{1}{2} |\bar{x}-\bar{y}|$.
Thus, from the above estimate we obtain
\begin{equation}\label{Lip-02}
L |\bar{x}-\bar{y}|\leq 2(u(\bar x) - u(\bar y)) \leq 2 M_u.
\end{equation}

Now we split the proof in two cases.

\medskip

\noindent
{\bf Case 1: $m>2s$.} 
In this case we will require the estimate in Proposition~\ref{P2.1}, so we recall $A$ from~\eqref{defA} and that 
$$
\gamma_0 := \frac{m - 2s}{m - 1}.
$$

Since the functions $t \mapsto \frac{1}{1 - t}$ and $t \mapsto \frac{m - t}{1 - t}$ are increasing in $t \in (0,1)$, we will provide a proof for a 
{ $t=\theta_0$  sufficiently small}. 

We consider $L$ in~\eqref{penalization} with the form
\begin{equation}\label{Rio1}
L = \widehat{C} C_{H,R+2}^{\frac{1}{1 - \theta_0}}\mathfrak{Z}^{\frac{m - \theta_0}{1 - \theta_0}},
\end{equation}
for some $\widehat{C} > 0$ which will be taken large just in terms of the data and $\theta_0$, with
$\theta_0$ taken suitably small depending on  $d$, $s$ and $m$. The constant $C_{H,R}$ appears in (H1), and $\mathfrak{Z} = \max \{ \mathfrak{Z}_1, \mathfrak{Z}_2 \}$, with 
\begin{equation}\label{Zeta1}
\mathfrak{Z}_1 = R^{1 - \gamma_0}(A^{1/m} + \mathrm{osc}_{B_{R + 2}} u)
\end{equation}
and 
\begin{equation}\label{Zeta2}
\mathfrak{Z}_2 = 1 + M_u + \mathrm{osc}_{B_{R + 2}} f + \mathscr{H}(u, B_{R + 2}).
\end{equation}

By Proposition~\ref{P2.1} there exists a constant $C$ such that
\begin{equation}\label{cotaHolder}
|u(x) - u(y)| \leq C\, \mathfrak{Z}_1 |x - y|^{\gamma_0} \quad \mbox{for all} \ x, y \in B_{R+\frac{3}{2}}.
\end{equation}

Combining the last estimate together with the first inequality in~\eqref{Lip-02}, 
we conclude that
\begin{equation}\label{cota1}
|\bar x - \bar y|^{1 - \gamma_0} \leq 2 C \mathfrak{Z}_1 L^{-1}.
\end{equation}

Set $\phi(x,y) = L\varphi(|x - y|) + M_u \psi(x)$, which will play the role of a test function touching $u$ at $\bar x$ (for subsolution's inequality), and at $\bar y$ (for supersolution's inequality). We use the viscosity inequalities, and write, for all $\delta\in ( 0, 1/4]$ that
\begin{align*}
     -I[B_\delta](\phi(\cdot, \bar y), \bar x) - I[B_\delta^c](u, \grad_x \phi(\bar x, \bar y), \bar x) + H(\bar x, L \bar p + M_u \grad\psi(\bar x)) + \alpha u(\bar{x}) -f(\bar{x})& \leq 0, \\
     -I[B_\delta](-\phi(\bar x, \cdot), \bar y) - I[B_\delta^c](u, -\grad_y \phi(\bar x, \bar y), \bar y) + H(\bar y, L \bar p) + \alpha u(\bar{y})-f(\bar{y})& \geq 0,
\end{align*}
where $\bar p = \varphi'(|\bar x - \bar y|) \frac{\bar x - \bar y}{\abs{\bar x - \bar y}}$.
Subtracting both inequalities we arrive at
\begin{align}\label{ineqvisc}
    H_1 \leq I_1 + I_2 + I_3 + I_4,
 \end{align}
where 
\begin{align*}
 I_1 &= I[B_\delta](\phi(\cdot, \bar y), \bar x) + I[B \cap B_\delta^c](u, \grad_x \phi(\bar x, \bar y), \bar x) 
      \\
 I_2 &= -I[B_\delta](-\phi(\bar x, \cdot), \bar y) - I[B\cap B_\delta^c](u, -\grad_y \phi(\bar x, \bar y), \bar y) 
     \\
I_3 &=  I[B^c](u, \bar x) - I[B^c](u, \bar y),
      \\
I_4 &= \alpha (u(\bar{y})-u(\bar{x})) + (f(\bar{x})-f(\bar{y})),
      \\
H_1 &= H(\bar x, L \bar p + M_u \grad\psi(\bar x)) - H(\bar y, L \bar p).
\end{align*}

Now we follow the approach of Barles, Chasseigne,  Ciomaga \& Imbert \cite{BCCI12} to treat the nonlocal terms. Denote by 
$\bar a = \bar x - \bar y$. For some $\delta_0\in (1, \frac{1}{4}], \eta_0 \in (0,1)$ to be fixed, we define 
$$
\mathcal C = \mathcal C_{\delta_0, \eta_0} := \{ z \in B_{\delta_0} : |z \cdot \bar a| \geq (1 -  \eta_0) |\bar a| |z| \}.
$$

Using the inequality 
$$
\Phi(\bar x, \bar y) \geq \Phi(\bar x + z, \bar y)
\Rightarrow u(\bar{x}+z)-u(\bar{x})\leq \phi(\bar{x}+z, \bar y)-\phi(\bar{x}, \bar y),
$$
 we obtain
\begin{equation*}
    I_1 \leq I[\mathcal C \cup B_\delta](\phi(\cdot, \bar y), \bar x) + I[B \cap (\mathcal C \cup B_\delta)^c](u, \grad_x \phi(\bar x, \bar y), \bar x),
\end{equation*}
and similarly, using $\Phi(\bar x, \bar y) \geq \Phi(\bar x, \bar y+ z)$ we get
$$
I_2 \leq -I[\mathcal C \cup B_\delta](-\phi(\bar x,  \cdot), \bar y) - I[B \cap (\mathcal C \cup B_\delta)^c](u, -\grad_y \phi(\bar x, \bar y), \bar y).
$$

Using that $\Phi(\bar x + z, \bar y + z) \leq \Phi(\bar x, \bar y)$ for $|z|\leq 1$, we also have
\begin{align*}
     & I[B \cap (\mathcal C \cup B_\delta)^c](u, \grad_x \phi(\bar x, \bar y), \bar x) - I[B \cap (\mathcal C \cup B_\delta)^c](u, -\grad_y \phi(\bar x, \bar y), \bar y) \\
     \leq & M_u I[B \cap (\mathcal C \cup B_\delta)^c] (\psi, \bar x).
\end{align*}

Thus combining the above estimate and the definition of $\phi$ we conclude that
\begin{align*}
I_1 + I_2 \leq & I[\mathcal C \cup B_\delta](L \varphi(|\cdot - \bar y|, \bar x) - I[\mathcal C \cup B_\delta](-L \varphi(|\bar x - \cdot|), \bar y) \\
& + M_u I[B](\psi, \bar x).
\end{align*}

Using a second-order Taylor expansion for $\psi$, the last term in the above inequality can be estimated as
\begin{align*}
I[B](\psi, \bar x) \leq \frac{1}{2} \| D^2 \psi  \|_{L^\infty(B_{\frac{1}{4}})} \int_{B} |z|^2 K(z) dz
\leq  C_2,
\end{align*}
where $C_2 > 0$ depends only on $d,s$ and $\Lambda$. 

At this point, we can send $\delta \to 0$, and by the smoothness of $\varphi, \psi$  we infer that
\begin{equation}\label{Lip-03}
    I_1 + I_2 \leq LI[\mathcal C](\varphi(|\cdot - \bar y|, \bar x) +L I[\mathcal{C}]( \varphi(|\bar x - \cdot|), \bar y) + M_u C_2.
\end{equation}

Since the first two terms in the right-hand side of the above inequality can be estimated in the same fashion and from which we obtain the same upper bound, we concentrate on the first one.
We follow closely the arguments of~\cite{BCCI12}. Applying Taylor's expansion, we note that
\begin{align*}
    I[\mathcal{C}](\varphi(|\cdot - \bar y|, \bar x) = \frac{1}{2}\int_{\mathcal C} \sup_{|t| \leq 1} \Big{\{} & \frac{\varphi'(|\bar a + tz|)}{|\bar a + tz|} (|z|^2 - \langle \widehat{\bar a + tz}, z \rangle^2) \\
    & + \varphi''(|\bar a + tz|)\langle \widehat{\bar a + tz}, z \rangle^2 \Big{\}} K(z)dz,
\end{align*}
where $\widehat{e}=\frac{1}{|e|}e$, the unit vector along the direction 
$e\in\rd\setminus\{0\}$.

We consider $\delta_0 = \delta_1 |\bar a|$ with $\delta_1 \in (0,1)$ to be fixed. Let $z \in \mathcal C$ and $t \in [-1,1]$. Then,
\begin{equation*}
    |\langle \bar a + tz, z \rangle| \geq (1 - \eta_0 - \delta_1)|\bar a||z|.
\end{equation*}

Similarly, we also have
\begin{equation*}
 (1 - \delta_1) |\bar a| \leq   |\bar a + tz| \leq (1 + \delta_1) |\bar a|.
\end{equation*}

Thus, using the fact that $\varphi' \geq 0$ and $\varphi'' \leq 0$ in the above integration, and denoting 
$$
\tilde \eta = \frac{1 - \eta_0 - \delta_1}{1 + \delta_1},
$$ 
we arrive at
\begin{align*}
    I[\mathcal C](\varphi(|\cdot - \bar y|, \bar x) \leq \frac{1}{2} \int_{\mathcal C} \sup_{|t|\leq 1} \Big\{  ( 1 - \tilde \eta^2)\frac{\varphi'(|\bar a + tz|)}{|\bar a|(1 - \delta_1)} + \tilde \eta^2 \varphi''(|\bar a + tz|) \Big\}
    |z|^2 K(z) dz.
\end{align*}

We set $\delta_1, \eta_0 \in (0,1)$ small enough so that
\begin{equation*}
    \frac{1 - \tilde \eta^2}{1 - \delta_1} \leq 4 (\eta_0 + 2 \delta_1), \ \tilde \eta^2 \geq 1/2.
\end{equation*}

By the definition of $\varphi$, we then conclude that
\begin{align*}
    I[\mathcal C](\varphi(|\cdot - \bar y|, \bar x) \leq \frac{1}{2} \int_{\mathcal C} [ & 4 (\eta_0 + 2\delta_1) |\bar a|^{-1} -2^{-2}\theta(1 + \theta) |\bar a|^{\theta - 1} ]|z|^2 K(z) dz.
\end{align*}

Finally, we consider $\eta_0 = \epsilon_0 |\bar a|^\theta, \delta_1 = \epsilon_0 |\bar a|^\theta$ for some $\epsilon_0 > 0$ small to be fixed, to arrive at
\begin{align*}
    I[\mathcal C](\varphi(|\cdot - \bar y|, \bar x) \leq \frac{12 \epsilon_0 -2^{-2}\theta(1 + \theta))|\bar a|^{\theta - 1}}{2} \int_{\mathcal C} |z|^2 K(z) dz.
\end{align*}

Thus, for all $\theta \in (0,1)$, there exists $\epsilon_0=\epsilon_0(\theta)$ small enough such that 
\begin{align*}
    I[\mathcal C](\varphi(|\cdot - \bar y|, \bar x) \leq -2^{-4}\theta(1 + \theta)|\bar a|^{\theta - 1} \int_{\mathcal C} |z|^2 K(z)dz,
\end{align*}
and by the lower bound in the ellipticity condition~\eqref{ellipticity}, we conclude that (see \cite[Example~1]{BCCI12})
\begin{align*}
    I[\mathcal C](\varphi(|\cdot - \bar y|, \bar x) \leq - c_\theta \lambda |\bar a|^{\theta - 1} \eta_0^{\frac{d - 1}{2}} \delta_0^{2 - 2s},
\end{align*}
for some constant $c_\theta>0$.
By the choice of $\delta_0, \eta_0$, we then obtain
\begin{equation*}
    I[\mathcal C](\varphi(|\cdot - \bar y|, \bar x) \leq -\tilde c_\theta |\bar a|^{\frac{\theta(d + 1)}{2} + 2(\theta + 1)(1 - s) - 1}, 
\end{equation*}
fo some $\tilde c_\theta
 > 0$ not depending on $L$. Putting this estimate in \eqref{Lip-03}, we get 
\begin{equation*}
I_1 + I_2 \leq - 2\tilde{c}_\theta L |\bar a|^{\frac{\theta(d + 1)}{2} + 2(\theta + 1)(1 - s) - 1} + C_2 M_u.
\end{equation*}

It is easy to see that
$$
I_3 \leq 2 \Lambda \sup_{x\in B_{R + 2}} \int_{B_1^c} |u(x+y)-u(x)|\frac{dy}{|y|^{d+2s}} \leq 
4 \Lambda \mathscr{H}(u, B_{R + 2}),
$$
and since $u(\bar{x})-u(\bar{y})>0$,
$$ 
I_4\leq \alpha (u(\bar y) - u(\bar x)) + \mathrm{osc}_{B_{R + 2}} f \leq \mathrm{osc}_{B_{R + 2}} f.
$$

Thus, gathering the above estimates, the right-hand side in \eqref{ineqvisc} can be estimated as
\begin{equation}\label{Lip-04}
I_1 + I_2 + I_3 + I_4 \leq -2 \tilde{c}_\theta L |\bar a|^{\theta((d + 1)/2 + 2(1 - s)) + 1 - 2s} + C_3 \mathfrak{Z}_2,
\end{equation}
for some constant $C_3 > 0$ depending only on $d,s$ and $\Lambda$, and $\mathfrak{Z}_2$ is given by \eqref{Zeta2}.

Since $2s>1$, we can fix $\theta$ small enough so that the exponent of
 $|\bar a|$ above is negative. We fix this $\theta$, and denoting $\tilde \theta = \theta((d + 1)/2 + 2(1 - s))$ for simplicity, 
from \eqref{cota1} we obtain  that
$$
L|\bar a|^{\tilde \theta + 1 - 2s} \geq L \tilde C \mathfrak{Z}_1^{\frac{\tilde \theta + 1 - 2s}{1 - \gamma_0}} L^{-\frac{\tilde \theta + 1 - 2s}{1 - \gamma_0}} = \tilde C \mathfrak{Z}_1^{\frac{\tilde \theta + 1 - 2s}{1 - \gamma_0}} L^{\frac{2s - \gamma_0 - \tilde \theta}{1 - \gamma_0}},
$$
for some $\tilde C$ depending only on $d, s, m, \Lambda$. Plugging this into~\eqref{Lip-04} we arrive at
\begin{equation}\label{Lip-045}
I_1 + I_2 + I_3 + I_4 \leq -2 \tilde{c}_\theta \tilde C \mathfrak{Z}_1^{\frac{\tilde \theta + 1 - 2s}{1 - \gamma_0}} L^{\frac{2s - \gamma_0 - \tilde \theta}{1 - \gamma_0}} + C_3 \mathfrak{Z}_2,
\end{equation}

\medskip
 
To estimate the Hamiltonian we use \hyperlink{H1}{(H1)} and~\eqref{Lip-02} to conclude that
\begin{align*}
    H_1 \geq & - C_{H, R+2} \left( (1+ L^m) |\bar x - \bar y| + (1+ L^{m-1} + (M_u|\grad\psi(\bar{x})|)^{m-1}) M_u |\grad\psi(\bar{x})|\right)
     \\
    \geq & -\tilde C_4 C_{H,R+2}(M_u L^{m-1} + (M_u)^m + 1), 
\end{align*}
for some constant $\tilde C_4 > 0$. Since by \eqref{firstchoice} we have
$$L^{m-1}M_u\geq \frac{1}{(\varphi(4^{-1}))^{m-1}} M_u^m,$$
we get from above that
\begin{equation}\label{Lip-05}
    H_1\geq - C_4 C_{H, R+2} (M_u L^{m-1}+1),
\end{equation}
for some constant $C_{4}$.

Replacing \eqref{Lip-045} and \eqref{Lip-05} into~\eqref{ineqvisc}, we obtain
\begin{equation*}
    2 \tilde{c}_\theta \mathfrak{Z}_1^{\frac{\tilde \theta + 1 - 2s}{1 - \gamma_0}} L^{\frac{2s - \gamma_0 - \tilde \theta}{1 - \gamma_0}} \leq C_3\mathfrak{Z}_2 + C_4 C_{H, R+2} (M_u L^{m-1}+1).
\end{equation*}

At this point, we notice that 
$$
\frac{2s - \gamma_0}{1 - \gamma_0} = m; \quad \frac{2s - 1}{1 - \gamma_0} = m - 1,
$$
from which, denoting $\theta_0 = \frac{\tilde \theta}{1 - \gamma_0}$, we can write the last inequality as
$$
\tilde c_{\theta} \mathfrak{Z}_1^{1 -m + \theta_0} L^{m -\theta_0} \leq C_3\mathfrak{Z}_2 + C_4 C_{H, R+2} (M_u L^{m-1}+1),
$$
for some $\tilde c_{\theta} > 0$. Since $\max \{ 1, M_u \} \leq \mathfrak{Z}_2$ we obtain
$$
\tilde c_{\theta} \mathfrak{Z}_1^{1 -m + \theta_0} L^{m - \theta_0} \leq C_5 C_{H, R+2} \mathfrak{Z}_2 L^{m-1},
$$
from which we arrive at
$$
\tilde c_\theta L^{1 - \theta_0} \leq C_5 \mathfrak{Z}_1^{m -1 - \hat\theta} C_{H, R+2} \mathfrak{Z}_2.
$$
Now, fix $\theta$ small enough such that  (see the definition $\tilde\theta$)
$$
{ \theta_0= \theta\frac{(d + 1)/2 + 2(1 - s)}{1-\gamma_0} < 1}.
$$
{ We choose $t=\theta_0$ as mentioned before}. Then, taking $\mathfrak{Z} = \max \{ \mathfrak{Z}_1, \mathfrak{Z}_2 \}$ we have
$$
\tilde c_{\theta} L^{1 - \theta_0} \leq C_6 C_{H,R+2} \mathfrak{Z}^{m - \theta_0},
$$
from which
$$
L \leq \bar C C_{H,R+2}^{\frac{1}{1 - \theta_0}}\mathfrak{Z}^{\frac{m - \theta_0}{1 - \theta_0}},
$$
for some $\bar C > 1$ just depending on the data and the choice of $\theta$. 

Then, by the choice~\eqref{Rio1}
we arrive at a contradiction by taking $\widehat{C}$ large enough in comparison to $\bar C$.  For instance, we can let $\widehat{C}=2\bar{C}$.
 Thus, taken {$\widehat{C}$} in this way ({which also fixes $L$}), we must have
$M\leq 0$. In particular,
\begin{equation*}
    \sup_{x,y \in B_{R + 2}} \{ u(x) - u(y) - L \varphi(|x - y|) - 
    M_u \psi(x) \} \leq 0,
\end{equation*}
implying, 
$$
u(x) - u(y) \leq L \varphi(|x - y|), \quad \mbox{for all} \ x, y \in  B_R.
$$

\medskip

\noindent
{\bf Case 2:  $m < 2s + 1$.} {Note that for $m\in (2s, 2s+1)$, the  proof is covered by Case 1, but the proof in this case does not require Proposition~\ref{P2.1}.}
In this case we do not need to use the 
H\"older regularity as in \eqref{cotaHolder}, and therefore from~\eqref{Lip-02} we can infer that
\begin{equation}\label{controlLcase2}
L|\bar x - \bar y| \leq 2M_u.
\end{equation}

The choice of $L$ takeing place of~\eqref{Rio1} in the previous case this time is
$$
L = \widehat{C} \mathfrak{Z}_2^{\frac{2s - \theta_0}{2s + 1 - m - \theta_0}} C_{H, R+2}^{\frac{1}{2s + 1 - m - \theta_0}},
$$
for some $\widehat{C} > 0$ large enough.

Following the same arguments as in the previous case, after localization and testing the problem at points $\bar x, \bar y$, we can estimate the nonlocal terms as before to conclude the same estimate~\eqref{Lip-04}. Estimates for $I_4$ and $H_1$ follow exactly the same lines as before, and we use~\eqref{controlLcase2} to arrive at
$$
c_1 L^{2s - \tilde \theta} M_u^{\tilde \theta + 1 - 2s} \leq C_2 L |\bar a|^{\tilde \theta + 1 - 2s} \leq C_3 \mathfrak{Z}_2 C_{H,R+2} L^{m-1},
$$
for some constants $c_1, C_2, C_3 > 0$ not depending on $L$ nor $R$. Then, cancelling terms we arrive at
\begin{equation*}
L^{2s + 1 - m - \tilde \theta} \leq \bar C M_u^{2s - 1 - \tilde \theta} \mathfrak{Z}_2 C_{H,R+2},
\end{equation*}
from which we arrive at the result by the choice of $\widehat{C}$ large enough compare to $\bar C$ in the definition of $L$ { and letting $\tilde\theta$ small enough by fixing $\theta$ small}. The proof is now complete.
\qed

\begin{rem}
{ Theorem~\ref{Thm-Lip} provides an upper bound on the gradient of $u$ in terms of $f$, ${\rm osc}(u)$ and $\mathscr{H}$. Typically, for Hamilton-Jacobi equations with
superlinear nonlinearity in the gradient, one uses the classical Bernstein method, introduced by Bernstein in \cite{Bern1,Bern2}, to estimate the gradient term. The main idea of this
technique is find an elliptic equation solved by $z(x):=\eta(x)|\grad u(x)|^{2}$, $\eta$ is a suitable cut-off function, and then obtain the gradient estimate applying maximum principle. Hence such
gradient estimates are also known as Bernstein estimate in literature. It is also worth pointing out that Bernstein method often produces {\it optimal} bound on the gradient.
We may see \eqref{estBernstein} as a Bernstein type estimate though our method of proof is very different from the Bernstein technique}. It is quite 
possible that the estimate may be
far from being optimal. On the other hand,
our estimate does not require $f$ to be Lipschitz continuous,
which is commonly used in this type of estimates \cite{CDV22,Ichi2011}. 
\end{rem}

As a consequence of Theorem~\ref{Thm-Lip} we obtain the following
regularity result.
\begin{thm}\label{T2.2}
Assume the hypotheses of Theorem~\ref{Thm-Lip}. Then for every 
$\bar{\eta}\in (0, 2s-1)$, the solution $u$ of \eqref{Lip-01} is in $C^{1, \bar\eta}_{loc}$.
In addition, let us assume that the kernel $K$ defining $I$ is symmetric and for each $R > 0$ and $\kappa > 0$, there exists $C_{R, \kappa} > 0$ such that 
\begin{equation}\label{KLip}
|K(z - x_1) - K(z - x_2)| \leq C_{R, \kappa} |z|^{-(d + 2s + 1)}|x_1 - x_2|, 
\end{equation}
for each $x_1, x_2 \in B_R$, and $z \in \rd$ such that $|z-x_1|, |z - x_2| \geq \kappa$.
Let $D_1 \subset \rd$ be a nonempty open set and assume $u \in C(\rd) \cap L^1(\omega_s)$ be a viscosity solution to
$$
\alpha u - I u + H(x,\grad u) - f(x) = 0 \quad \mbox{in} \ D_1,
$$
with $f$ locally H\"older continuous in $D_1$.
Then, $u\in C^{2s+\eta}_{loc}(D_1)$ for some $\eta>0$.
\end{thm}

\begin{proof}
Since $u$ is locally Lipschitz by Theorem~\ref{Thm-Lip}, given any
bounded domain $D$ it is easily seen that
$$-C \leq \int_{\rd} (u(x+y)-u(x)-y\cdot \grad u(x)) K(y)dy \leq C
\quad \text{in}\; D,$$
for some constant $C$. We also note that the kernel $K$ satisfies the conditions of \cite[(A1)-(A4)]{SS2016}. Now the
proof of $C^{1, \bar\eta}_{loc}$ follows from a scaling
argument introduced by Serra in \cite[Theorem~2.2]{Serra-15a} in combination
to the H\"{o}lder estimate from \cite[Theorem~7.2]{SS2016}. See also 
\cite[Theorem~5.2]{BK22} for a similar argument.

Next we prove the $C^{2s+\eta}$ regularity.
Let $D_2, D_3$ be smooth bounded domains such that
$D\Subset D_2\Subset D_3\Subset D_1$.  From Theorem~\ref{Thm-Lip}
we see that $u$ is Lipschitz in $D_3$.
Consider a smooth cut-off function
$\chi:\rd\to [0, 1]$ satisfying $\chi=1$ in $D_2$ and 
$\chi=0$ in $D^c_3$. Define
$$
h(x)=\int_{\rd}u(x+y)(1-\chi(x+y))K(y) dy.
$$

Thus, letting $v=u\chi$, we can write the above equation as
\begin{equation}\label{ET2.2A}
-I u=F(x):= f(x) - H(x, \grad u) - h(x)-\alpha u\quad \text{in}\; D_2.
\end{equation}

Since, $u$ is Lipschitz in $D_2$, we have 
$$\sup_{x\in D_2}|F(x)|<\infty.$$
Applying \cite[Theorem~4.1]{K2013} we then see that $\grad u$ is locally H\"{o}lder continuous in $D_2$. Now consider a smooth domain
$D^\prime$ so that $D\Subset D^\prime\Subset D_2$. It then follows that
$f, H(x, \grad u(x))$ and $u$ are H\"{o}lder continuous in $D^\prime$.
By the property of $\chi$, we can find $\kappa>0$ so that
$$h(x)= \int_{|y|\geq \kappa}u(x+y)(1-\chi(x+y))K(y) dy.$$

We claim that
$$x\mapsto \int_{|y|\geq \kappa} u(x+y) K(y)dy$$
is Lipschitz in $D^\prime$. To see this,
consider $x_1, x_2\in D^\prime$ with $|x_1-x_2|\leq\delta$. Then,
for $\delta\leq\kappa/4$, in view of~\eqref{KLip},  we can write
\begin{align*}
&\left|\int_{B^c_\kappa} u(x_1+y) K(y)dy-\int_{B^c_\kappa} u(x_2+y) K(y)dy \right|
\\
&= \left|\int_{B^c_\kappa(x_1)} u(z) K(z - x_1)dz-\int_{B^c_\kappa(x_2)} u(z) K(z - x_2)dz \right|
\\
&\leq  \int_{(B_\kappa(x_1) \cup B_\kappa(x_2))^c} |u(z)| |K(z-x_1)- K(z-x_2)| \dz 
\\
&\qquad +\int_{B_\kappa(x_2) \setminus B_\kappa(x_1)} |u(z)| K(z - x_1)dz + \int_{B_\kappa(x_1) \setminus B_\kappa(x_2)} |u(z)| K(z-x_2)dz
\\
&\leq C |x_1-x_2| \int_{|z|\geq \kappa/2} |u(z)| |z|^{d+2s + 1}dz + C\max_{D^\prime} |u| \kappa^{-(d + 2s)} |x_1-x_2|,
\end{align*}
for some constant $C > 0$ depending on $\Lambda, d, s, D'$ and $\kappa$. 
This proves the claim on Lipschitz regularity. Since $\chi$ is smooth,
we have $h$ Lipschitz in $D^\prime$. Hence $F$ in~\eqref{ET2.2A} is H\"{o}lder
continuous in $D^\prime$.
Thus, applying the regularity result in \cite[Theorem~1.3]{S2015} we see that $u\in C^{2s+\eta}(D)$ for some $\eta>0$.
\end{proof}


\section{The general set up. Proof of existence of the eigenpair.}
\label{secgeneral}

Here and in what follows, we adopt the notation
$$
\cL u := -I u  + H(x,\grad u),
$$
where $Iu$ is given by \eqref{gen-I}, that is,
\begin{equation*}
I u(x) = \int_{\rd} (u(x+y)-u(x)-\chi_{B}(y)\, y\cdot \grad u(x))K(y)dy.
\end{equation*}

We recall $f$ satisfies assumptions~\hyperlink{F1}{(F1)}--\hyperlink{F2}{(F2)}. We need to restrict a bit assumptions~\hyperlink{H1}{(H1)}-\hyperlink{H2}{(H2)} in an uniform fashion as follows
\begin{itemize}
\item[\hypertarget{H1'}{(H1')}] There exists $m > 1$ and $C > 0$ such that
\begin{equation*}
|H(x, p+q)-H(y, p)|\leq C \left[ |x-y| (1+|p|^m) + |q| (|p|^{m-1}+|q|^{m-1})\right]
\end{equation*}
for all $ x, y, p, q\in\rd$.

\item[\hypertarget{H2'}{(H2')}] There exists a constant 
$\upkappa$ such that
$$\frac{1}{\upkappa}|p|^m-\upkappa \leq H(x, p) \leq \upkappa(1+|p|^m)
\quad \text{for all}\; x, p\in\rd.$$
\end{itemize}

Note that \hyperlink{H1'}{(H1')} and \hyperlink{H2'}{(H2')} imply \hyperlink{H1}{(H1)} and \hyperlink{H2}{(H2)}, respectively, and therefore, the regularity results
developed in Section~\ref{secLip} also hold under the assumptions \hyperlink{H1'}{(H1')} -\hyperlink{H2'}{(H2')}.
We also introduce the following condition
\begin{itemize}
\item[\hypertarget{H3}{(H3)}] There exist positive constants $C, b_m$ such that for all $\mu\in(0,1)$
\begin{equation}\label{EH2}
\mu H(x, \mu^{-1}p) - H(x, p)\geq (1-\mu) (b_m|p|^m-C)\quad \text{for all}\; x, p\in\rd.
\end{equation}
\end{itemize}

Since $f$ is coercive and $H(x, 0)$ is bounded in view of
\hyperlink{H2'}{(H2')}, we can add a suitable constant to
$f$ to assure that
\begin{equation}\label{E3.3}
H(x, 0)-f(x)\leq 0\quad \text{in}\; \rd.
\end{equation}

The structure of \eqref{E3.1} allows for this modification. 

The main result of this paper in its general form is the following
\begin{thm}\label{T1.2}
Let $K$ satisfies~\eqref{ellipticity}. 
Suppose $H$ satisfies \hyperlink{H1'}{(H1')}, \hyperlink{H2'}{(H2')} and~\hyperlink{H3}{(H3)}; and that $f$ satisfies
\hyperlink{F1}{(F1)}--\hyperlink{F2}{(F2)}.

Then, there exists  $(u, \lambda^*) \in C(\rd) \times \R$ solving the ergodic problem
\begin{equation*}
\cL u= f-\lambda^*\quad \text{in}\; \rd.
\end{equation*}
The ergodic constant $\lambda^*$ meets the characterization formula
\begin{equation}\label{lam-crit}
\lambda^*=\inf\{\lambda\; : \; \exists \, u\in C_+(\rd)\cap L^1(\omega_s)\; 
\text{satisfying}\; \cL u\geq f-\lambda\; \text{in}\; \rd\}.
\end{equation}
and $u$ is unique up to an additive constant.
Provided $x\mapsto H(x, p)$ is periodic, we have
\begin{equation}\label{ET1.2B}
\lambda^*=\sup\{\lambda\; : \; \exists \, u\in C_{-}(\rd)\cap L^1(\omega_s)\; 
\text{satisfying}\; \cL u\leq f-\lambda\; \text{in}\; \rd\}.
\end{equation}
Moreover, if $f$ is locally H\"older continuous and the nonlocal kernel
$K$ is symmetric satisfying \eqref{KLip},
then $u \in C^{2s^+}(\rd)$. In particular, $u$ is a classical solution.
\end{thm}
{ Note that the Hamiltonian $H$ in Theorem~\ref{TeoIntro} does not depend on $x$ whereas Theorem~\ref{T1.2} includes a class of Hamiltonian depending on $x$. For 
the representation \eqref{ET1.2B} we require $H$ to be periodic in $x$, that is, for a period $T>0$ we must have $H(x, p)=H(x+Te_i, p)$ for all $x, p\in\rd$ where $e_i$ denotes that
$i$-th canonical unit basis vector in $\rd$ for $i=1, 2,\ldots, d$. The proof of \eqref{ET1.2B} is based on certain approximation argument based on periodic functions which uses 
the existence result from \cite{BKLT15}. For more detail see Theorem~\ref{T4.2} below.}

In the next subsection we provide the existence part of Theorem~\ref{T1.2}.

\subsection{Existence for the ergodic problem.}
We firstly provide an existence result of an eigenpair
$(u, \tilde\lambda) \in C(\rd)\cap L^1(\omega_s) \times \R$ satisfying
\begin{equation}\label{E3.1}
\cL u = f-\tilde\lambda \quad\text{in}\; \rd,
\end{equation}
and $u\geq 0$.

\medskip

The proof of existence of an eigenpair $(u, \lambda)$ is obtained in two main steps, that we can summarize as follows:
\begin{itemize}
\item[$(i)$] For $\alpha > 0$, solve the discounted problem
\begin{equation}\label{E3.4}
\mathcal{L}u - f +\alpha u =0 \hspace{3mm} \text{in}\hspace{2mm}
\mathbb{R}^d.
\end{equation}

This is the content of Proposition~\ref{T3.2}. We denote $w_\alpha$ the solution for this problem. 

\medskip

\item[$(ii)$] We show $\bar{w}_\alpha=w_\alpha-w_\alpha(0)$ is precompact in uniform norm on compact sets as $\alpha \to 0$, which leads us to the existence for the ergodic problem~\eqref{E3.1}. This is the content of Theorem~\ref{T3.5} below and covers the existence part in Theorem~\ref{T1.2}.
\end{itemize}

As we shall see below, in both the steps, passage of limits inside the integration
are justified using the barrier function $V$ constructed in the next lemma.
\begin{lem}\label{lemmaV}
Assume~\hyperlink{H2'}{(H2')} and \hyperlink{F1}{(F1)}. Let $\beta\in (1, 2s)$ be such that $m(\beta  - 1) > \gamma$, and consider a smooth, nonnegative function $V$ such that $V(x) = |x|^\beta$ for $|x| \geq 1$. Then, there exist $\upkappa_0, \upkappa_1, R_0 > 0$ such that
\begin{equation}\label{Supsol}
\mathcal L V(x) - f(x) \geq -\upkappa_0 \chi_{B_{R_0}} + \upkappa_1 |x|^{m(\beta - 1)} \quad \mbox{in} \ \rd.
\end{equation}
\end{lem}

\begin{proof}
We first claim that 
\begin{equation}\label{V-1}
|IV(x)|\leq C(1+|x|^{\beta-1})\quad \text{in}\; \rd.
\end{equation}
By the smoothness of $V$, we have the existence of $C_1 > 0$ such that $|I V(x)| \leq C_1$ for all $|x|\leq 4$. So we consider $|x|> 4$ and note that
\begin{align*}
|IV(x)|& \leq \Lambda \int_{|y|\leq 2} |V(x+y)-V(x)-\chi_B y\cdot \grad V(x)|
|y|^{-d-2s} dy 
\\
&\quad + \Lambda \int_{|y|> 2} |V(x+y)-|x|^\beta-\beta\chi_B y\cdot x|x|^{\beta-2}|
|y|^{-d-2s} dy
\\
&\leq C_1 + \Lambda 
\underbrace{\int_{|y|> 2} ||x+y|^{\beta}-|x|^\beta-\beta\chi_B y\cdot x|x|^{\beta-2}| |y|^{-d-2s} dy}_{=J_1}
\\
&\quad + \underbrace{\Lambda \int_{|y|> 2} |V(x+y)-|x+y|^\beta||y|^{-d-2s} dy}_{=J_2}.
\end{align*}
$J_1$ can be computed as follows (recall that $\hat{x}=x/|x|$)
\begin{align*}
J_1&=|x|^{\beta-2s}\int_{|y|>\frac{2}{|x|}} 
||\hat{x}+y|^{\beta}-1-\beta\chi_{B_{\frac{1}{|x|}}} y\cdot \hat{x}| |y|^{-d-2s} dy
\\
&\leq |x|^{\beta-2s}\left[\int_{\frac{2}{|x|}\leq |y|\leq 1} 
||\hat{x}+y|^{\beta}-1-\beta\chi_{B_{\frac{1}{|x|}}} y\cdot \hat{x}| |y|^{-d-2s} dy + C_2\right]
\\
&\leq |x|^{\beta-2s} \left[C_3 + \beta \int_{\frac{2}{|x|}\leq |y|\leq 1} 
|y|^{-d-2s+1} dy + C_2\right]\leq C_4 (1+ |x|^{\beta-1}), 
\end{align*}
for some appropriate constants $C_2, C_3, C_4$. To compute $J_2$ we note that
for $|y|\leq \frac{|x|}{2}$, we have $|x+y|\geq |x|-\frac{|x|}{2}=\frac{|x|}{2}>1$. Hence, for $|y|\leq \frac{|x|}{2}$, we have $V(x+y)=|x+y|^\beta$.
Therefore
\begin{align*}
J_2&=\int_{|y|> \frac{|x|}{2}} |V(x+y)-|x+y|^\beta||y|^{-d-2s}
\\
& \leq C_5 \int_{|y|> \frac{|x|}{2}} (1+ |y|^\beta) |y|^{-d-2s} dy
\leq C_6 |x|^{\beta-2s},
\end{align*}
for some constants $C_5, C_6$. Combining these estimates we obtain \eqref{V-1}.
Using \eqref{V-1} we can write
\begin{align*}
\mathcal{L}V(x)- f(x) &\geq -C(1+ |x|^{\beta-1})+ H(x, \grad V(x))
- C(1+|x|^\gamma)\nonumber
\\
&\geq -C(1+ |x|^{\beta-1})+ C(|x|^{m(\beta-1)}-1)
- C(1+|x|^\gamma)\nonumber
\\
&\geq -\upkappa_0 \chi_{B_{R_0}} +\upkappa_1 |x|^{m(\beta-1)}
\end{align*}
where $R_0 > 1$ is large enough, and 
$\upkappa_0, \upkappa_1$ are positive constants.
\end{proof}

\subsubsection{Existence for the $\alpha$-discounted problem.}
Let $\alpha \geq 0$. For each $n \in \mathbb N$, consider the Dirichlet problem
\begin{equation}\label{EBD-1}
\left \{ 
\begin{array}{rll}
\cL W_n - f + \alpha W_n=&0 \quad \text{in}\hspace{2mm} B_n, 
\\
W_n=&0 \quad \text{in}\hspace{2mm} B^c_n,
\end{array}
\right .
\end{equation}
where $B_n$ is the open ball of radius $n$, centered at the origin. 

Problem~\eqref{EBD-1} has a unique viscosity solution $W_n \in C(\rd)$, which can be obtained by Perron's method. It is easy to see that the function identically equal to zero is a viscosity subsolution for the Dirichlet problem, in view of~\eqref{E3.3}. On the other hand, it is direct to construct a supersolution to the problem with the form
$$
\varphi(x) = \min \{ C_1 (n - |x|)_+^\beta, C_2 \},
$$
where $\beta \in (0,s)$ and $C_1, C_2 > 0$ are suitable constants. In fact, by Lemma 3.1 in~\cite{DQT-17} or Proposition 4.8 in~\cite{DRSV22}, there exists $c_0 > 0$ such that $-I \varphi(x) \geq c_0(n - |x|)^{\beta- 2s}$ for all $x$ close to the boundary, provided $\beta$ is small enough (depending on the ellipticity constants). Thus, we fix $C_1, C_2$ large in terms of $n$ in order $\varphi$ is a (continuous) viscosity supersolution to the problem attaining the boundary condition.

Thus, the existence of a continuous solution is completed by the following comparison  principle:

\begin{lem}[Comparison]
\label{comparison}
Let $\alpha > 0$, $f \in C(\rd)$ and 
assume H satisfies \hyperlink{H1'}{(H1')}, \hyperlink{H2'}{(H2')} and \hyperlink{H3}{(H3)}.
Let $D$ be a bounded domain with a smooth boundary, and let $\alpha \geq 0$. Let $u\in USC(\bar{D}) \cap C(D^c) \cap L^1(\omega_s)$ be a viscosity subsolution to
\begin{align*}
\cL u - f + \alpha u= 0  \hspace{2mm} \text{in} \hspace{2mm} D,
\end{align*}
and $v\in LSC(\bar{D})\cap C(D^c) \cap L^1(\omega_s)$ be a viscosity supersolution to 
\begin{align*}
\cL v - f + \alpha v= 0  \hspace{2mm} \text{in} \hspace{2mm} D.
\end{align*}

If $v\geq u$ in $D^c$, then we have $v\geq u$ in $\mathbb{R}^d$. 
\end{lem}
\begin{proof}
By contradiction, suppose that $\sup_{\rd}(u-v)=\sup_D(u-v)>0$. From the semicontinuity of $u, v$ we can find $x_0\in D$ such that 
$$
M:= (u-v)(x_0)=\max_{\bar D}(u-v)=\sup_{\rd}(u-v)>0.
$$

Since $u, v \in L^1(\omega_s)$, for $R \geq \mathrm{diam}(D) + 1$ large enough in terms of $M$, $\alpha$ and the ellipticity constants, we have the function $\tilde u = u \chi_{B_R}$ satisfies
$$
\cL \tilde u - f +\alpha \tilde u \leq \alpha M/4 \quad \mbox{in} \ D,
$$
and $\tilde v = v \chi_{B_R}$ satisfies
$$
\cL \tilde v - f +\alpha \tilde v \geq -\alpha M/4 \quad \mbox{in} \ D,
$$
in the viscosity sense. Thus, without loss of generality we assume that $u$ and $v$ are bounded, and satisfy the above mentioned inequalities.

Then for $\varepsilon \in (0,1)$ close to $1$ we have
\begin{equation}\label{Com-1}
\sup_{D^c}(\varepsilon u-v) \leq (1-\varepsilon)\sup_{D^c}|u|
< (\varepsilon u-v)(x_0).
\end{equation}

Letting $w_\varepsilon=\varepsilon u-v$ it follows from \cite[Lemma~3.2]{BKLT15} that
\begin{equation*}
-I w_\varepsilon- \tilde{C}\frac{(m-1)^{m-1}}{m^m}\frac{|\grad w_\varepsilon|}{[(1-\varepsilon)b_m]^{m-1}} + \alpha w_\varepsilon \leq C(1-\varepsilon) + (1 + \epsilon)\alpha M/4
\end{equation*}
in the viscosity sense,
where $C$ is the constant given by \eqref{EH2} and $\tilde{C}$ is a constant dependending on $H$ but not on $\varepsilon$. From \eqref{Com-1} we can find
$x_\varepsilon\in D$ so that 
$$
M_\epsilon := (\varepsilon u- v)(x_\varepsilon)= \sup_{\rd}(\varepsilon u-v)>0.
$$

Using the constant function equal to $M_\varepsilon$ as test function for $w_\varepsilon$ at $x_\varepsilon$, we arrive at
$$
\alpha w_\varepsilon(x_\varepsilon) \leq C(1 - \varepsilon) + (1 + \varepsilon)\alpha M/4.
$$

Since
$$\lim_{\varepsilon\nearrow 1} w_\varepsilon(x_\varepsilon)=
\max_{\rd}(u-v)=(u-v)(x_0) = M,$$
letting $\varepsilon\to 1$ in the previous inequality we get
$$
\alpha M \leq \alpha M/2,
$$
which is a contradiction with the fact that $M > 0$. 
\end{proof}

Now we can prove the existence result for the discounted problem~\eqref{E3.4}, which covers the step $(i)$ mentioned at the beginning of this section. A special attention to be given on the approximation procedure which we use to obtain the results. Similar argument is going to be useful in subsequent analysis.
\begin{prop}\label{T3.2}
Assume hypotheses of Theorem~\ref{T1.2} hold, and for each $n$ let $W_n$ be the unique viscosity solution  to \eqref{EBD-1}. Then, $\{ W_n \}_n$ converges (up to subsequences) locally uniformly in $\rd$ to a function $w_\alpha \in C(\rd)$ solving~\eqref{E3.4}.

Furthermore, $0\leq w_\alpha\leq \frac{\upkappa_0}{\alpha} + V$ in $\rd$.
\end{prop}

\begin{proof}
Let $W_n$ be a solution to \eqref{EBD-1} and define $\tilde{V}(x):= \frac{\upkappa_0}{\alpha}+ V(x)$, where $\upkappa_0$ is given by \eqref{Supsol}.
Observe from \eqref{Supsol} that
\begin{align*}
\mathcal{L} \tilde{V} - f + \alpha \tilde V\geq 0. 
\end{align*}
Therefore, by the comparison principle Lemma~\ref{comparison}, we have $W_n\leq \tilde{V}$ and $W_n\geq 0$ in $\mathbb{R}^d$ for all 
$n\in \mathbb{N}$. Thus for any compact set $\mathscr{B}$,
${\rm osc}_{\mathscr{B}}W_n$ is uniformly bounded in $n$. Similarly, we also
have
$$
\sup_{n\geq n_0} \, \sup_{x\in \mathscr{B}} \int_{|y|\geq \frac{1}{4}} |W_n(x+y)-W_n(x)|\frac{1}{|y|^{d+2s}} \dy<\infty,
$$
where $n_0$ is chosen large enough so that $\mathscr{B}\Subset B_{n_0}$.
Thus, by Theorem~\ref{Thm-Lip}, the sequence $\{W_n\, :\, n\geq 1\}$
is uniformly locally Lipschitz.
Now we apply Arzel\'a-Ascoli and use a diagonalization process to get a subsequence $\{W_{n_k}\}$ converging to $w_\alpha\in C(\mathbb{R}^d)$ uniformly on every compact sets. As we already have $0\leq W_{n_k}\leq \tilde{V}$ in whole $\mathbb{R}^d$, it implies 
\begin{align*}
 0\leq w_\alpha\leq \tilde{V} \hspace{2mm} \text{in} \hspace{2mm} \mathbb{R}^d.
\end{align*}
Moreover, we can apply dominated convergence theorem and conclude
\begin{align*}
\norm{W_{n_k}-w_\alpha}_{L^1(\omega_s)} \xrightarrow{k\rightarrow \infty} 0.
\end{align*}
In the end we use stability of the viscosity solution
to \eqref{EBD-1} to obtain a solution to~\eqref{E3.4}. This completes the proof.
\end{proof}

\subsubsection{Existence for the ergodic problem}
In this section we show that $w_\alpha(\cdot)-w_\alpha(0)$ converges to a 
solution to the ergodic problem. To do so, we need to find a bounded of this
function. 
We begin with the following lemma.
\begin{lem}\label{L3.3}
Let \hyperlink{H2'}{(H2')} and \hyperlink{F2}{(F2)} hold.
Suppose that $\alpha\in [0, 1]$ and assume $w\in C(\rd)\cap L^1(\omega_s)$ with $\inf_{\rd} w>-\infty$ is a viscosity supersolution to~\eqref{E3.4}. Then $w$ is coercive, that is, 
$\lim_{|x|\to \infty} w(x)=\infty$.
\end{lem}

\begin{proof}
The proof relies on comparison principle. Let $x_0\in \rd$ and define
$$\psi=\mu \max\{(1-|x-x_0|^2), -3\},$$
for some $\mu>1$ to be chosen later. It is easily seen that
$\psi\in C^2(B_{2}(x_0))$ and $\psi\geq -3\mu$. Thus,
$$|\cL \psi + \alpha \psi|\leq \mu^m \kappa,\quad \text{in}\; B_{\sqrt{2}}(x_0) $$
for some $\kappa$, independent of $x_0$. 
Let $\mu=\mu(x_0)=
[\frac{1}{2\kappa} \min_{B_{\sqrt{2}}(x_0)} f]^{\nicefrac{1}{m}}$. Since $f$
is coercive, for $|x_0|$ large, we obtain
\begin{align*}
\cL \psi - f + \alpha \psi& < 0\quad \text{in}\; B_{\sqrt{2}}(x_0),
\\
\psi\leq -\mu&\leq w\quad \text{in}\; B^c_{\sqrt{2}}(x_0).
\end{align*}
Since it is a strict subsolution, the comparison principle applies and we have $w\geq \psi$. In particular, $w(x_0)\geq \psi(x_0)=\mu(x_0)$.
Hence, using the definition of $\mu$, we get the desired result.
\end{proof}

Next lemma basically contains the expected property mentioned in point $(ii)$ at the beginning of this section. 
\begin{lem}\label{L3.4}
Let \hyperlink{H1'}{(H1')},\hyperlink{H2'}{(H2')} and 
\hyperlink{F1}{(F1)}-\hyperlink{F2}{(F2)} hold.
For each $\alpha \in (0,1)$, let $w_\alpha$ in Proposition~\ref{T3.2} and denote 
$$
\bar{w}_\alpha(x):=w_\alpha(x)-w_\alpha(0).
$$

There exists $R > 0$ such that
\begin{equation}\label{EL3.4A}
-\max_{\bar B_R} |\bar{w}_\alpha|\leq \bar{w}_\alpha(x)\leq \max_{\bar B_R} |\bar{w}_\alpha| + V(x)\quad x\in\rd,
\end{equation}
where $V$ is given by Lemma~\ref{lemmaV}. Moreover, we have
\begin{equation}\label{EL3.4B}
\sup_{\alpha\in (0, 1)}\, \max_{\bar B_R} |\bar{w}_\alpha| <\infty.
\end{equation}
\end{lem}

\begin{proof}
 Firstly, since $w_\alpha\geq 0$,
 using Lemma~\ref{L3.3} we see that $w_\alpha$ coercive. Then, let $x_\alpha\in \Argmin_{\rd} w_\alpha$.
Using $\varphi\equiv w_\alpha(x_\alpha)$ as a test function for $w_\alpha$ at $x_\alpha$, we get
$$\cL \varphi(x_\alpha) - f(x_\alpha)+\alpha w_\alpha(x_\alpha)\geq 0,$$
implying 
$$f(x_\alpha)-\upkappa\leq f (x_\alpha)-H(x_\alpha, 0)\leq \alpha w_\alpha(x_\alpha)\leq
\alpha w_\alpha(0)\leq \upkappa_0 + V(0),$$
where we have used Proposition~\ref{T3.2} and the fact that $w_\alpha(x_\alpha)\leq w_\alpha(0)$. Since $f$ is coercive, there exists $R > 0$ such that $x_\alpha\in B_R$ for all
$\alpha$. Thus, for all $x \in \rd$ we can write
\begin{equation*}
\bar{w}_\alpha(x)\geq w_\alpha(x_\alpha)-w_\alpha(0)\geq
-\sup_{B_R} |\bar{w}_\alpha|.
\end{equation*}

This gives the l.h.s. inequality of \eqref{EL3.4A}. To establish the r.h.s., we first observe that 
$$
\limsup_{|x| \to \infty} \frac{w_\alpha(x)}{V(x)} = 0.
$$

Indeed, taking
$0<\beta'<\beta$ satisfying $m(\beta'-1)>\gamma$,
 and considering $\hat V: \rd \to \R$ smooth, nonnegative, with $\hat{V}=|x|^{\beta'}$ for $|x| > 1$, the calculation in Lemma~\ref{lemmaV} follows along the same lines, and the the argument in Proposition~\ref{T3.2} gives $|w_\alpha|\leq \frac{\kappa_1}{\alpha} +
 \hat{V}$ for some suitable $\kappa_1$. Thus 
 $w_\alpha\in o(V)$. In particular, $V-w_\alpha$ is
 coercive. Let $z_\alpha\in \Argmin_{\rd} (V-w_\alpha)$. We can use $\varphi(x)=V(x)-V(z_\alpha)+ w_\alpha(z_\alpha)$ as
a test function touching $w_\alpha$ at $z_\alpha$ from above.
From~\eqref{Supsol}, we then have
\begin{align*}
0&\geq \cL V(z_\alpha) - f(z_\alpha)+\alpha w_\alpha(z_\alpha)
\\
&\geq -\upkappa_0+\upkappa_1 |z_\alpha|^{m(\beta-1)} - \alpha(V(z_\alpha)-
w_\alpha(z_\alpha)) + \alpha V(z_\alpha)
\\
&\geq -\upkappa_0 + \upkappa_1 |z_\alpha|^{m(\beta-1)} -
\alpha(V(0)-
w_\alpha(0)).
\end{align*}
Using the fact that $V, w_\alpha$ are nonnegative, we arrive at
$$
\upkappa_1 |z_\alpha|^{m(\beta-1)}\leq \upkappa_0 + V(0),
$$
which in turn, implies that there exists $R > 0$ such that
 $z_\alpha\in B_R$ for all $\alpha$. In particular,
$$
\bar{w}_\alpha(x)\leq V(x) - V(z_\alpha) + w_\alpha(z_\alpha)-w_\alpha(0)\leq V(x)+\sup_{B_{R}}|\bar{w}_\alpha|.
$$
This gives us the r.h.s. in~\eqref{EL3.4A}.

Next we deal with~\eqref{EL3.4B}.
Suppose, on the contrary, that  there
exists a sequence $\alpha_n\searrow 0$ such that $\tau_n := \sup_{B_R} |\bar{w}_{\alpha_n}|\to \infty$ as $n\to\infty$. Define
$$
v_n(x)=\frac{1}{\tau_n} \bar{w}_{\alpha_n}(x) \quad \mbox{for} \ x \in \rd.
$$

We have the following

\medskip

\noindent
{\bf\textit{Claim:}} The family $\{ v_{n} \}_n$ is locally equicontinuous.

\medskip

Let us first complete the proof assuming the Claim. We write \eqref{E3.4}
as
$$
- I v_n + \frac{1}{\tau_n} H(x, \tau_n \grad v_n) - \frac{1}{\tau_n} f(x)+\alpha_n v_n=-\frac{1}{\tau_n}\alpha_n w_{\alpha_n}(0) \quad \mbox{in} \ \rd.
$$

Using \hyperlink{H2'}{(H2')} we then obtain
\begin{equation}\label{EL3.2C}
- I v_n + \upkappa^{-1}\tau_n^{m-1} |\grad v_n|^m - \frac{\upkappa}{\tau_n} - \frac{1}{\tau_n} f(x)+\alpha_n v_n + \frac{1}{\tau_n}\alpha_n w_{\alpha_n}(0)\leq 0 \quad \text{in}\; \rd,
\end{equation}
in the viscosity sense. Moreover, \eqref{EL3.4A} gives
\begin{equation}\label{EL3.2D}
-1\leq v_n\leq 1+ \frac{1}{\tau_n} V\quad \text{in}\; \rd.
\end{equation}
From the Claim, we can extract a subsequence of $\{v_n\}_n$ that converges locally uniformly to some $v\in C(\rd)$ which, in view of \eqref{EL3.2D}, satisfies
\begin{equation}\label{EL3.2E}
|v|\leq 1 \quad \text{in}\; \rd,\quad v(0)=0,\quad \text{and}\quad \max_{\bar B_R}|v|=1.
\end{equation}

Moreover, by Proposition~\ref{T3.2} we have $\{\alpha_n w_{\alpha_n}(0)\}_n$ is bounded. Therefore, by stability and \eqref{EL3.2C} we have 
$v$ satisfying
\begin{equation*}
|\grad v|\leq 0\quad \text{in}\; \rd,
\end{equation*}
in the viscosity sense. This implies that $v$ is a constant. In fact, for each $x_0 \in \rd$ and all $\varepsilon > 0$, the function $\varphi_\varepsilon(x) = v(x_0)+\varepsilon |x-x_0|$ is a viscosity solution to
$$
|\grad u| \geq \varepsilon \quad \mbox{in} \ \rd \setminus \{ x_0 \}. 
$$
Now, since $\|v\|_\infty \leq 1$, for each $\varepsilon > 0$ there exists $R_\epsilon > 0$ such that $\varphi_\varepsilon \geq v$ in $B_{R_\epsilon}^c \cup \{ x_0 \}$, and since $v$ is a strict subsolution for the problem for which $\varphi_\varepsilon$ is a viscosity supersolution, together with the convexity of the Hamiltonian $p \mapsto |p|$, comparison principle holds (here the regularity of the boundary does not play any role). Hence, $v \leq \varphi_\varepsilon$ in $\rd$ for all $\epsilon$. Taking $\varepsilon \to 0$, we get
$v(x)\leq v(x_0)$ for all $x\in\rd$. Since $x_0$ is an arbitrary point, $v$ must be a constant, which contradicts~\eqref{EL3.2E}. Thus \eqref{EL3.4B} must hold.

\medskip

We finish the proof of the lemma by providing the

\noindent
{\bf \textit{Proof of the Claim:}} In the supercritical case, that is $m>2s$, the proof follows from the H\"{o}lder regularity result of \cite{BKLT15}. We provide a unified proof, motivated
from \cite{BCCI14}. Let
$$
f_n:= -\frac{1}{\tau_n} f(x)+\alpha_n v_n+\frac{1}{\tau_n}\alpha_n w_{\alpha_n}(0).
$$

Then, $\{ f_n \}_n$ is uniformly bounded on compact sets, and $\{ v_n \}_n$ satisfies 
\begin{equation}\label{EL3.2G}
- Iv_n + \frac{1}{\tau_n} H(x, \tau_n \grad v_n) + f_n(x)=0\quad \text{in}\; \rd.
\end{equation}
Also, from \eqref{EL3.2D}, $\{v_n\}$ is locally uniformly bounded.
Given $R > 1$, consider a smooth function $\psi:\rd\to [0, 1]$ which is radially increasing, satisfying $\psi=0$ in $B_R$
and $\psi=1$ in $B^c_{R+\frac{1}{2}}(0)$. Let
$$ 
M= 2\, \sup_{n}\, \sup_{B_{R + 1}} |v_n| < +\infty.
$$
Note that it is enough to show that for each $\eta \in (0,1)$ and for each $\epsilon\in (0, 1)$ there exists  $L > 0$ such that
\begin{equation}\label{EL3.2H}
v_n(x)-v_n(y)\leq L |x-y|^\eta + 2M(1-\epsilon) + M \psi(x) \quad \forall\; x, y\in B_{R + 1}, \ n \in \mathbb N.
\end{equation}
We can conclude a uniform modulus of continuity for the family $\{ v_n \}_n$ in $B_R$ from above.

\vspace{.2in}

Define
$$
\Phi(x, y)= \epsilon v_n(x)-v_n(y) - L|x-y|^\eta-M\psi(x)\quad \text{in}\; B_{R+1}\times B_{R + 1}.
$$
It is easily seen that if $\Phi(x, y)\leq 0$ in $ B_{R + 1}\times B_{R+1}$, then we have \eqref{EL3.2H}. So suppose that
$$
\sup_{B_{R + 1}\times B_{R + 1}}\Phi >0.
$$
Now, by the definition of $\psi$ and $M$, we have $\Phi(x, y)<0$ for $x\in B^c_{R+\frac{1}{2}}(0)$. Thus, if we set $L$ large enough to satisfy
$$L(\nicefrac{1}{4})^\eta> 2M,$$
then we also have $\Phi(x, y)<0$ for $y\in B^c_{R+\frac{3}{4}}(0)$. Thus, there exists points $\bar{x}, \bar{y}\in B_{R + 1}$ (for brevity we drop the dependence on $n$ here) so that
$$
\sup_{B_{R + 1}\times B_{R + 1}}\Phi(x, y)=\Phi(\bar{x}, \bar{y})>0.
$$

If $\bar{x}=\bar{y}$, then 
$$
\Phi(\bar{x}, \bar{y})\leq (\epsilon-1)v_n(\bar{x})\leq M (1-\epsilon),
$$
which in turn gives us \eqref{EL3.2H}, so we are left with the situation when $\bar{x}\neq \bar{y}$. Since 
$\Phi(\bar{x},\bar{y})\geq \Phi(\bar{x}, \bar{x})$, we obtain
\begin{equation}\label{EL3.2I}
L|\bar{x}-\bar{y}|^\eta \leq v_n(\bar{x})-v_n(\bar{y})\leq M,
\end{equation}
and the choice of $L$ above implies that $|\bar{x}-\bar{y}|<\frac{1}{4}$. Denote by
$$p=L\eta |\bar{x}-\bar{y}|^{\eta-2}(\bar{x}-\bar{y}), \quad q=\grad\psi(\bar{x}).$$
From \eqref{EL3.2I}, it follows that
\begin{equation}\label{EL3.2J}
|p|=L \eta|\bar{x}-\bar{y}|^{\eta-1}\geq M^{\frac{\eta-1}{\eta}} L^{\frac{1}{\eta}}\eta.
\end{equation}

Define $\phi(x, y)= L|x-y|^\eta + M\psi(x)$. Recalling the notation~\eqref{Ieval} , we use 
the viscosity inequalities for $v_n$ solving~\eqref{EL3.2G} to write, for each $\delta\in (0, \frac{1}{4})$ that
\begin{equation}\label{EL3.2K}
\begin{split}
-I[B_\delta](\phi(\cdot, \bar{y}), p + Mq, \bar{x}) - I[B^c_\delta](\epsilon v_n, p+Mq, \bar{x}) + \epsilon\tau^{-1}_n H(\bar{x}, \epsilon^{-1} \tau_n(p+Mq)) + \epsilon f_n(\bar{x})&\leq 0,
\\
-I[B_\delta](\phi(\bar{x}, \cdot), p, \bar{y}) - I[B^c_\delta](v_n,p, \bar{y}) +  \tau^{-1}_n H(\bar{y},  \tau_n p) +  f_n(\bar{y})&\geq 0.
\end{split}
\end{equation}
We subtract both inequalities and proceed to estimate the terms arising there. Concerning the contribution of the Hamiltonian, using \hyperlink{H3}{(H3)} and \hyperlink{H1'}{(H1')} we see that
\begin{align*}
& \epsilon H(\bar x, \epsilon^{-1} \tau_n (p + Mq)) - H(\bar y, \tau_n p) \\
= & \epsilon H(\bar x, \epsilon^{-1} \tau_n (p + Mq)) - H(\bar x, \tau_n (p + Mq)) + H(\bar x, \tau_n (p + Mq)) - H(\bar y, \tau_n p) \\
\geq & (1 - \epsilon) (b_m |\tau_n (p + Mq)|^m - C) \\
& - C_{H,R+1} \Big{(}|\bar x - \bar y|(1 + |\tau_n p|^m)
+ |\tau_n Mq|(1 + |\tau_n p|^{m - 1} + |\tau_n Mq|^{m - 1})\Big{)}.
\end{align*}
But $|q|$ is bounded, meanwhile, by~\eqref{EL3.2I} and~\eqref{EL3.2J}, $|p| \to \infty$, $|\bar x - \bar y| \to 0$ as $L \to +\infty$. Thus, by taking $L$ large enough, depending on $1 - \epsilon, M, R$ and the data, we conclude that
\begin{equation*}
\epsilon H(\bar x, \epsilon^{-1} \tau_n (p + Mq)) - H(\bar y, \tau_n p) \geq \frac{1}{2}(1 - \epsilon) b_m \tau_n^m |p|^m - C_\epsilon,
\end{equation*}
for some constant $C_\epsilon > 0$ not depending on $n$. 

Concerning the contribution of the non local terms, 
using $\Phi(\bar{x}, \bar{y})\geq \Phi(\bar{x}+z, \bar{y}+z)$
in $B_{R + 1}$, we see that
\begin{align*}
I[B_{\frac{1}{4}}\cap B^c_\delta](\epsilon v_n, p+Mq, \bar{x})-
I[B_{\frac{1}{4}}\cap B^c_\delta](v_n, p, \bar{y})
\leq M I[B_{\frac{1}{4}}\cap B^c_\delta](\psi, q, \bar{x})
\leq C M,
\end{align*}
for some $C > 0$ not depending on $n$, nor $\epsilon$. Again, using \eqref{EL3.2D}, we see that
$$\sup_{z\in B_{R+1}} |I[B^c_{\frac{1}{4}}](u, p, z)|\leq C_1 (1+|p|),$$
for some constant $C_1$, not depending on $n$.

Using these estimates and letting $\delta\to 0$, we conclude from the difference of the inequalities in~\eqref{EL3.2K} that 
\begin{equation*}
\frac{1}{2}(1 - \epsilon) b_m \tau^{m - 1} |p|^m \leq C_\epsilon + CM + 2 \|f_n \|_{L^\infty(B_{R + 1})} + C_1 (1+|p|).
\end{equation*}
Thus, in view of \eqref{EL3.2J} and since $m>1$, taking $L$ large enough in terms of $\epsilon$ and $R$, we arrive at a contradiction.
Thus $\bar{x}\neq \bar{y}$ is not possible for all
large $L$. This completes the proof of Claim, and hence the proof of the lemma.
\end{proof}


Now we are in position to provide the following existence result.
\begin{thm}\label{T3.5}
Suppose that \hyperlink{H1'}{(H1')}, \hyperlink{H2'}{(H2')} and
\hyperlink{F1}{(F1)}--\hyperlink{F2}{(F2)} hold. Then, there exists a solution $(u, \tilde \lambda) \in C(\rd) \times \R$ for~\eqref{E3.1}. 
Moreover, $u$ is coercive. In addition, if $f$ is locally H\"{o}lder
continuous and $K$ is symmetric satisfying \eqref{KLip}, then $u\in C^{2s+}(\rd)$.
\end{thm}

\begin{proof}
Let $\bar w_\alpha$ as in Lemma~\ref{L3.4}. 
Then, by~\eqref{EL3.4A}, there exists $\tilde \lambda \in \R$, $u: \rd \to \R$ with $\inf_{\rd}u>-\infty$, and a sequence $\alpha_k \to 0$ such that $\bar{w}_{\alpha_{k}}\to u\in C_{loc}(\rd)\cap L^1(\omega_s)$, and 
$$
\tilde\lambda = \lim_{k\to\infty}\alpha_{k}w_{\alpha_k}(0).
$$
By stability of the viscosity solution, we have $(u, \tilde \lambda)$ that solves~\eqref{E3.1}. This leads to the existence of a solution. 

The function $u$ is coercive by Lemma~\eqref{L3.3}, and 
$C^{2s+}$ regularity follows from Theorem~\ref{T2.2} when $f$ is locally  H\"older.
\end{proof}

\subsection{Nonexistence.} We finish this section with the following
\begin{thm}\label{T-nonex}
Let \hyperlink{H2'}{(H2')} hold and $I=-(-\Delta)^s$.
Suppose that for some constants $\kappa_0, \kappa_1$ we have
\begin{equation}\label{ET1.1A}
f(x)\geq \kappa_0 |x|^{m(2s-1)}\quad \text{for}\; |x|\geq \kappa_1.
\end{equation}

Then, there exists no pair $(u, \lambda) \in (C(\rd) \cap L^1(\omega_s)) \times \R$ with $u\geq 0$ that satisfy
\begin{equation*}
    \cL u - f+\lambda\geq 0 \quad \text{in}\; \rd.
\end{equation*}

In particular, recalling $\lambda^*$ in~\eqref{lam-crit}, we have $\lambda^*=\infty$.
\end{thm}

\begin{proof}
Since $f-\lambda$ would satisfy a similar lower bound
as in \eqref{ET1.1A}, we may assume that $\lambda=0$.
Let $\eta\in (0, 1)$ and $|x_0| > 1$ large enough. Define $\theta=\frac{1}{2}|x_0|$, and 
$$
\psi(x)=\eta \tilde c[(\theta^2-|x-x_0|^2)_+^s- \frac{3^s}{4^s}\theta^{2s}]\quad \text{in}\; B_\theta(x_0).
$$

Here $\tilde c = \tilde c(\theta, s, d) > 0$ is a suitable constant so that $(-\Delta)^s\psi=\eta$ in 
$B_\theta(x_0)$. Using \hyperlink{H2'}{(H2')}, we have in $B_{\theta/2}(x_0)$
that
\begin{align*}
    \cL\psi(x) - f(x) &\leq \eta +\upkappa + 
    \upkappa (c_{d,s}\eta)^m s^m (\theta^2-|x-x_0|^2)^{m(s-1)} 2^m|x-x_0|^m - f(x)
    \\
    &\leq \eta + \upkappa + \kappa_2 \eta^m  |\theta|^{m(2s-1)}- \kappa_3 
    |x_0|^{m(2s-1)},
\end{align*}
for some constants $\kappa_2, \kappa_3$, where we used the fact that $|x|\geq |x_0|-|x-x_0|\geq \frac{3}{4}|x_0|$ in $B_{\frac{\theta}{2}}(x_0)$.
Thus,  if we fix $\eta$ small enough, so that $\kappa_22^{-m}\eta^m<\frac{1}{2}\kappa_3$, 
 we get
$$\cL \psi - f< \eta + \upkappa-\frac{\kappa_3}{2}|x_0|^{m(2s-1)}
<0 \quad B_{\frac{\theta}{2}}(x_0),$$
for all $|x_0|$ large. Since $\psi\leq 0$ in $B^c_{\frac{\theta}{2}}(x_0)$, by the comparison principle we get $u\geq \psi$. Evaluating
at $x=x_0$ we have $u(x_0)\geq \eta c_{d,s} (1-\frac{3^s}{4^s})\theta^{2s}\geq \kappa_4 |x_0|^{2s} $ for some constant $\kappa_4$ and all
 $|x_0|$ large. Then $u\notin L^1(\omega_s)$, which is a contradiction. Then, for each $\lambda \in \R$ there is no function $u \in C(\rd) \cap L^1(\omega_s)$ such that $\mathcal L u \geq f - \lambda$ in $\rd$, from which $\lambda^* = +\infty$. This completes the proof.
\end{proof}

\begin{rem}
We do not know if the above nonexistence result is possible 
under a condition weaker than~\eqref{ET1.1A}, namely
$$
\limsup_{|x| \to \infty} \frac{f(x)}{|x|^{m(2s - 1)}} > 0.
$$
\end{rem}




\section{Characterization of $\lambda^*$.} 
\label{seccharact}

We start with the characterization~\eqref{lam-crit}.
\begin{prop}\label{T4.1}
Grant the setting of Theorem~\ref{T3.5}, and let $\tilde \lambda$ be the constant found there. Then,
$\tilde\lambda=\lambda^*$ where $\lambda^*$ is given by \eqref{lam-crit}.
\end{prop}

\begin{proof}
From \eqref{lam-crit} it follows that $\lambda^*\leq \tilde\lambda$.
Suppose that $\lambda^*< \tilde\lambda$. Therefore, by definition,
we can find $\lambda\in(\lambda^*, \tilde\lambda)$ and $v\in C(\rd)\cap L^1(\omega_s), v\geq 0$, satisfying
\begin{equation*}
\cL v - f + \lambda\geq 0 \quad \text{in}\; \rd.
\end{equation*}

Let $(u, \tilde\lambda)$ be the solution obtained in Theorem~\ref{T3.5} (with approximating sequence $\{ \bar w_{\alpha_k} \}$). 
Fix $M>0$. Since $\alpha_{k}w_{\alpha_k}\to \tilde\lambda$, we can then find $k_0 \in \mathbb N$ such that
$$
\lambda< \alpha_{k} w_{\alpha_{k}}(0) - \alpha_{k} M
 \quad \text{for all}\; k \geq k_0.
$$
We fix any such $k \geq k_0$, denote $\alpha_{k}$ by $\tilde{\alpha}$, and $\bar w_{\alpha_{k}}$ by $\tilde{w}$. Recall from 
Theorem~\ref{T3.2} that there exists a sequence of $\{W_l\}$ satisfying 
\begin{equation}\label{ET4.1B}
\begin{split}
\cL W_l -f + \tilde\alpha W_l=&0 \hspace{3mm} \text{in}\hspace{2mm} B_l,
\\
W_l=&0 \hspace{3mm} \text{in}\hspace{2mm} B^c_l,
\end{split}
\end{equation}
such that $W_l\to w_{\alpha_k}$ in $ L^1(\omega_s)$
and uniformly on compacts, as $l\to\infty$. Moreover,
due to \eqref{E3.3} we also have $W_l\geq 0$. Denote by
$\zeta_l(x)=W_l(x)-W_l(0)+M$. It then follows from \eqref{ET4.1B} that
\begin{equation}\label{ET4.1C}
\begin{split}
\cL\zeta_l(x) - f(x)+\tilde\alpha \zeta_l=&-\tilde\alpha W_l(0) +
\tilde\alpha M \hspace{3mm} \text{in}\hspace{2mm} B_l,
\\
\zeta_l=&-W_l(0)+M\leq M \hspace{3mm} \text{in}\hspace{2mm} B^c_l.
\end{split}
\end{equation}
For $l$ large enough we see that $\lambda< \tilde\alpha W_l(0) -\tilde\alpha M$. Again, by Lemma~\ref{L3.3}, $v(x)\to\infty$, as $|x|\to \infty$. Thus, for all large $l$ , $v$ is a strict supsersolution to
\eqref{ET4.1C}. Thus,
by comparison principle Lemma~\ref{comparison}, we have $v\geq \zeta_l$ in $\rd$.
Now letting $l\to\infty$, we obtain $w_{\alpha_{k}}(x)-w_{\alpha_{k}}(0) + M\leq v$ for all $k\geq k_0$. Letting 
$k\to \infty$ we obtain from Theorem~\ref{T3.5} that $u+M\leq v$ in $\rd$. But $u$ is a fixed function and $M$ is arbitrary. Therefore, such $v$ can not exist. Hence $\tilde{\lambda}=\lambda^*$.
\end{proof}

In the remaining part of this section, we prove \eqref{ET1.2B}.
\begin{thm}\label{T4.2}
Suppose that \hyperlink{H1'}{(H1')},\hyperlink{H2'}{(H2')}
\hyperlink{F1}{(F1)}--\hyperlink{F2}{(F2)} hold and $x\mapsto H(x, p)$ is periodic. Then we have $\lambda^*$ satisfying characterization~\eqref{ET1.2B}.
\end{thm}

The following easy lemma provides one side of the estimate.

\begin{lem}\label{L4.3}
Suppose that for some $\lambda\in\R$ there exists a $w\in C(\rd)\cap L^1(\omega_s)$ satisfying
$$ 
\cL w  \leq f- \lambda \quad \text{in}\;\rd,
$$
and $\sup_{\rd}w<\infty$. Then we must have $\lambda\leq \lambda^*$.
\end{lem}

\begin{proof}
Since $w$ is bounded from above, we can add an appropriate scalar to $w$ so that $w\leq 0$ in $\rd$.
Now suppose, on the contrary, that $\lambda-\lambda^*>\delta>0$ for some $\delta>0$. Recall the solution $u$ from Theorem~\ref{T3.5} and 
Proposition~\ref{T4.1}.
Since $u$ is coercive we have  $\lim_{|x|\to\infty} u=\infty$. We claim that
for any $k>0$ we must have $u-k\geq w$ in $\rd$, which is not possible for $k$ large enough. Therefore, we must have $\lambda\leq \lambda^*$.

To establish that claim, fix $k>0$ and define $\bar{u}=u-k$. Note that
$$
\cL \bar{u}= f-\lambda^*\geq f-\lambda+\delta\quad \text{in}\; \rd.
$$
Since $w\leq 0$, we can find a bounded domain $D$ satisfying $\min_{D^c}(\bar{u}-w)>0$. Now we choose $\kappa>0$ small enough so that
\begin{align*}
\cL w - f+ \lambda +\kappa w &\leq 0 \quad \text{in}\; D,
\\
\cL \bar{u} - f + \lambda +\kappa \bar{u} &\geq \frac{\delta}{2} \quad \text{in}\; D.
\end{align*}
Since $\bar{u}\in L^1(\omega_s)$, for all $N$ large we have
$$\cL \bar{u}_N - f + \lambda +\kappa \bar{u}_N \geq \frac{\delta}{4} \quad \text{in}\; D,$$
where $\bar{u}_N=\min\{\bar{u}, N\}$. Applying the comparison
principle Lemma~\ref{comparison} we then have $\bar{u}_N\geq w$
in $\rd$. Letting
$N\to\infty$,
 we conclude that $\bar{u}\geq w$ in $\rd$. This proves the claim.
\end{proof}
Now we can complete the
\begin{proof}[Proof of Theorem~\ref{T4.2}]
Let us denote the right hand side of~\eqref{ET1.2B} by $\Lambda$. From Lemma~\ref{L4.3} it is evident that $\Lambda\leq\lambda^*$.
So we have to show that $\Lambda \geq \lambda^*$. Without any loss of
generality, we may assume that $f\geq 0$ and $H$ is $1$-periodic in the $x$ variable, that is,
$$H(x+e_i, p)=H(x, p)\quad \text{for all}\; x, p\in \rd,$$
and for all canonical unit vectors $e_i$ in $\rd$.

Let $g:\mathbb{N}\to \mathbb{N}$ be a function satisfying the following
\begin{itemize}
\item $g(k)$ is strictly increasing to infinity.
\item  For the cube $Q_k$, centered at $0$ and side length $2g(k)$
we have $f>k$ in $Q^c_k$.
\end{itemize}
Define $F_k(x) = \min \{ f(x), k \}$ on $Q_k$. By construction, we have 
$F_k=k$ on $\partial Q_k$. We extend $F_k$ to all of $\rd$ in a periodic fashion. At this point, we note that 
both $F_k$ and $H$ are periodic in the $x$ variable with period $2g(k)$.
Applying \cite[Proposition~4.2]{BKLT15} we can find a unique periodic function $w_k$ satisfying
\begin{equation}\label{ET4.2AA}
\cL w_k - F_k + \varrho_k=0\quad \text{in}\; \rd
\end{equation}
for some $\varrho_k\in\R$.

\medskip

\noindent
{\bf \textit{Claim 1.}} $0\leq \varrho_k\leq \lambda^*$.

Since, by definition, $f\geq F_k$ for all $k$, it follows from 
\eqref{ET4.2AA} that
$$
\cL w_k - f + \varrho_k\leq 0 \quad \text{in}\; \rd.
$$

Applying Lemma~\ref{L4.3} we then have $\varrho_k\leq\lambda^*$. This gives one side of
the estimate.
Now suppose, on the contrary, that $\varrho_k<0$. Let 
$z\in \Argmin_{Q_k} w_k=\Argmin_{\rd} w_k$. 
Applying the definition of viscosity solution
to \eqref{ET4.2AA} and using the test function
$\varphi\equiv w_k(z)$ we get that 
$$0\leq \cL\varphi(z) - F_k(z)+\varrho_k=-F_k(z)+\varrho_k<0,$$
which is not possible. Hence $\varrho_k\geq 0$. This concludes the Claim 1.

\medskip

\noindent
{\bf \textsl{Claim 2.}} \textsl{There exists $R > 0$, independent of $k$,
such that 
$$
\min_{\bar B_R} w_k = \min_{\rd} w_k\quad \text{for all $k$ large}.
$$
}

Since $w_k$ is a periodic and continuous function, it attains its minimum at some point $z \in \rd$. Using the test function
$\varphi\equiv w_k(z)$ we see that
$$
- F_k(z)+\varrho_k \geq \cL\varphi(z) - F_k(z)+\varrho_k\geq 0,
$$
which together with Claim 1 lead us to
$$
F_k(z)\leq \varrho_k\leq\lambda^*.
$$

Thus, for $k>\lambda^*$, we necessarily have $F_k(z)=f(z)$, from which
it is easily seen that $z$ belongs to the compact set
$\{x: f(x)\leq \lambda^*\}$. This proves Claim 2.

\medskip

Now recall the smooth positive function $V$ from \eqref{Supsol}
 satisfying $V(x)=|x|^\beta$ 
for $|x|\geq 1$, where $1<\beta <2s$ and $m(\beta-1)>\gamma$. 
By adding a suitable constant we can assume $\min_{\bar B_R} w_k=0$.
We claim that for some $R' > R$ we have 
\begin{equation}\label{ET4.2B}
\min_{\bar B_{R'}}w_k=0\leq w_k(x)\leq \max_{\bar B_{R'}} w_k + V(x) \quad \text{in}\; \rd.
\end{equation}

The left hand side inequality is obvious from the above construction. So
 we only prove the right hand side estimate. Since $V-w_k$ is coercive,
it attains its minimum at some point $z_k \in \rd$. Then, using $V$ as test function for $w_k$ at $z_k$ together with~\eqref{Supsol}, we arrive at
$$
\upkappa_1 |z_k|^{m(\beta - 1)} - \upkappa_0 \chi_{B_{R_0}}(z_k) + f(z_k) - F_k(z_k) + \rho_k \leq 0, 
$$
and since $F_k \leq f$, and $\rho_k \geq 0$, we conclude that for all $k \in \mathbb N$ we have
$$
|z_k| \leq (\upkappa_0 \upkappa_1^{-1})^{\frac{1}{m(\beta - 1)}}.
$$
Letting $R'$ the right-hand side of the above inequality we conclude~\eqref{ET4.2B}.

Now, through minor modifications of the arguments  of Lemma~\ref{L3.4} it can be shown that 
$$
\sup_{k\in\mathbb{N}} \max_{B_{R'}} w_k < +\infty.
$$
Therefore, by Theorem~\ref{Thm-Lip}, we have the family 
$\{w_k\}_k$ is locally equibounded and equi-Lipschitz continuous. Then, there exists a pair $(v, \rho) \in C(\rd) \cap L^1(\omega_s) \times \R$ such that (up to subsequences that we do not relabel) $w_{k} \to v$ locally uniformly in $\rd$, and $\varrho_k \to \varrho$ as $k \to \infty$, and
$$
\cL v - f + \varrho=0\quad \text{in}\; \rd.
$$
Also, notice that $0 \leq v$ in $\rd$. Then, by \eqref{lam-crit} 
we have $\varrho\geq \lambda^*$, and using Claim 1, we get that $\varrho= \lambda^*$. Again, by the
definition of $F_k$ and \eqref{ET4.2AA} we see that $w_k$ is a viscosity solution to
$$
\cL w_k - f+\varrho_k\leq 0\quad \text{in}\; \rd,
$$
and since $w_k$ is bounded above we conclude that $\varrho_k\leq \Lambda$ for all $k$.
This of course, gives $\lambda^*\leq\Lambda$, completing the proof.
\end{proof}


\section{Uniqueness result and monotonicity of the critical eigenvalue}
\label{secunique}

In this section we prove uniqueness of $u$, obtained in Theorem~\ref{T3.5}, up to an addition by a scalar. As a consequence, we get a monotonicity property for the additive eigenvalue with respect to the source term $f$, see Theorem~\ref{T1.3} below.

Toward these goals we need the following key lemma which shows that the solution $u$ found in Theorem~\ref{T3.5} may be regarded as the \textit{minimal} solution.
\begin{lem}\label{L5.1}
Assume hypotheses of Theorem~\ref{T3.5} are in force, and let $u$ be the solution found there. There exists $R > 0$ such that, for any function $v \in C^1(\rd) \cap L^1(\omega_s)$ 
bounded from below, solving the inequality
$$
\cL v - f +\lambda^*\geq 0\quad \text{in}\; \rd
$$
in the viscosity sense, and satisfying $u\leq v$ in $B_R$, we have
$$
u \leq v \quad \mbox{in} \; \rd.
$$
\end{lem}

\begin{proof}
Without loss of generality, we can assume $f \geq 0$ and $\lambda^* \geq 0$ (see \eqref{E3.3}).
Let $R > 0$ be such that $\min_{x \in B_R^c} \{ f(x) - \lambda^*-C \} > 1$
where $C$ is given by \eqref{EH2}. So $R$ depends only on the data. 

Let $v$ be such that $u\leq v$ in $B_R$.
Suppose, on the contrary,  that $u > v$ at some point in $\bar{B}_R^c$, say $x_0\in B^c_R$. Then, there exists $\mu \in (1/2, 1)$ sufficiently close to $1$, so that
\begin{equation}\label{EL5.6A}
\max_{\bar{B}_R}(\mu u - v)< \mu u(x_0)-v(x_0) \quad \text{and}\quad \mu u(x_0)-v(x_0)>0.
\end{equation}

Now, by a standard diagonalization method and the proof of Theorem~\ref{T3.5}, we can find a sequence 
of $(W_{n_k}, \alpha_{n_k})$ such that 
$$
\left \{ \begin{array}{rll} \cL W_{n_k} - f + \alpha_{n_k} W_{n_k} &=0\quad & \text{in}\; B_{n_k}, \\
W_{n_k} &=0\; & \text{in}\; B^c_{n_k}, \end{array} \right .
$$
with
\begin{align}
\lim_{n_k\to \infty} \alpha_{n_k} W_{n_k}(0)&=\lambda^*,\label{EL5.6AA2}
\\
\lim_{n_k\to\infty} \bar{W}_{n_k} &= u\quad \text{uniformly over compacts},\label{EL5.6AA3}
\end{align}
where $\bar{W}_{n_k}(x):=W_{n_k}(x)-W_{n_k}(0)$. We choose $n_k$ large enough so that $x_0\in B_{n_k}$. Using \eqref{EL5.6A}
and the above limit, we further choose $n_k$ so that
\begin{equation}\label{EL5.6B}
\max_{\bar{B}_R}(\mu \bar{W}_{n_k} - v)< \mu \bar{W}_{n_k}(x_0)-v(x_0),
\quad \mu \bar{W}_{n_k}(x_0)-v(x_0)>0,\quad v>0 \; \text{in}\; B^c_{n_k},
\end{equation}
where the last inequality is possible due to Lemma~\ref{L3.3}.
Define $\widehat{v}_{n_k}(x)=v+\mu W_{n_k}(0)$. It is easily seen that
\begin{equation*}
\cL \widehat{v}_{n_k} - f + \lambda^*\geq 0\quad \text{in}\; \rd,
\end{equation*}
and, since $W_{n_k}=0<v$ in $B^c_{n_k}$, it follows from \eqref{EL5.6B} that 
$$\kappa_{n_k}:=\sup_{\rd}(\mu W_{n_k}-\widehat{v}_{n_k})=\sup_{B_{n_k}}(\mu W_{n_k}-\widehat{v}_{n_k})>0.$$
Furthermore, the above supremum is attained at some point $x_{n_k}\in B_{n_k}\setminus B_R$. Thus,
$$\widehat{v}_{n_k}\geq \mu W_{n_k} -\kappa_{n_k}\quad \text{in}\; \rd,$$
and equality holds at the point $x_{n_k}$. By Theorem~\ref{T2.2} and
our hypothesis we have $u, v\in C^1(\rd)$ and therefore,
$x\mapsto H(x, \grad u(x)), H(x, \grad \widehat{v}(x))$ are continuous. Hence
we have 
$$-I (\widehat{v}-\mu W_{n_k}) + H(x, \grad \widehat{v})-\mu H(x, \mu^{-1}\grad W_{n_k})
- (1-\mu)f(x)+\lambda^* - \mu\alpha_{n_k} W_{n_k}\geq 0\quad \text{in}\; B_{n_k},$$
in the viscosity sense. This can obtained by mimicking the arguments
of \cite[Lemma~3.2]{BKLT15}. Evaluating at the point $x_{n_k}$,
where $\grad \widehat{v}(x_{n_k})=\mu\grad W_{n_k}(x_{n_k})$, we obtain
\begin{align*}
0&\leq  -\mu H(x_{n_k}, \grad W_{n_k}(x_{n_k})) + H(x_{n_k}, \mu\grad W_{n_k}(x_{n_k})) + (\mu-1) f(x_{n_k}) +  \lambda^* 
-\mu \alpha_{n_k} W_{n_k}(x_{n_k})
\end{align*}
which, in turn, gives from \eqref{EH2} 
\begin{align}\label{EL5.6D}
(\mu-1) f(x_{n_k}) +  \lambda^*-\mu \alpha_{n_k} W_{n_k}(x_{n_k})
&\geq \mu H(x_{n_k}, \grad W_{n_k}(x_{n_k}))- H(x_{n_k}, \mu\grad W_{n_k}(x_{n_k}))\nonumber
\\
&\geq - (1-\mu) C.
\end{align}
Now we have two possibilities. Suppose $|x_{n_k}|\to \infty$, along some subsequence, as $n_k\to\infty$. Since $W_{n_k}\geq 0$
and $f$ is coercive, we get a contradiction from \eqref{EL5.6D}. So the other possibility is $\{x_{n_k}\}$ is a bounded sequence.
Without loss of generality, we may assume that $x_{n_k}\to \tilde{x}$. It follows from the property of $\{x_{n_k}\}$ that
$|\tilde{x}|\geq R$. Using \eqref{EL5.6AA2} and \eqref{EL5.6AA3} we see that
$$
\alpha_{n_k} W_{n_k}(x_{n_k})=\alpha_{n_k} \bar{W}_{n_k}(x_{n_k}) + \alpha_{n_k} W_{n_k}(0)\to \lambda^*
$$
as $n_k\to\infty$. Thus, from \eqref{EL5.6D} we obtain
$$(1-\mu) (\lambda^*+ C-f(\tilde{x}) )\geq 0
\quad \text{and}\quad |\tilde{x}|\geq R.$$
This is also contradictory to the choice of $R$. Hence we must have $u\leq v$ in $\rd$.
\end{proof}

Now we are ready to establish the uniqueness result.

\begin{thm}\label{T5.2}
Suppose that \hyperlink{H1'}{(H1')}, \hyperlink{H2'}{(H2')}, \hyperlink{F1}{(F1)}-\hyperlink{F2}{(F2)} hold.
Let $v\in C^{1}(\rd)\cap L^1(\omega_s)$ be a non-negative supersolution satisfying
$$\cL v - f +\lambda^*\geq 0\quad \text{in}\; \rd,$$
and $u$ be the solution obtained in Theorem~\ref{T3.5}. Then we have $v=u+c$ for some scalar $c$.
\end{thm}

\begin{proof}
Let  $B_R$ be the ball given by Lemma~\ref{L5.1}. Define 
$$\min_{\bar{B}_R}(v-u)=c.$$
Letting $v_1=v-c$, we note that 
$$\cL v_1 - f +\lambda^*\geq 0\quad \text{in}\; \rd,$$
and $v_1$ touches $u$ in $\bar{B}_R$ from above.
Using Lemma~\ref{L5.1} it follows that $v_1\geq u$ in $\rd$. Setting 
$\varphi=v_1-u$ we see that
$$-I \varphi - H(x, \grad v_1)+ H(x, \grad u)\geq 0\quad \text{in}\; \rd.$$
Since $\inf_{\rd}\varphi=\min_{\bar{B}_R}\varphi=0$, there exists 
$\bar{x}\in \bar{B}_R$ such that $\varphi(\bar{x})=0=\grad\varphi(\bar{x})$.
This of course, implies $\grad u(\bar{x})=\grad v_1(\bar{x})$. So applying
$\psi\equiv \varphi(\bar{x})$ as the test function at the point $\bar{x}$
gives
$$-I[B_\delta^c](\varphi, 0, \bar x)\geq 0.$$
Since $\delta$ is arbitrary, we 
must have $\varphi\equiv 0$, implying $v=u+c$. Hence the proof.
\end{proof}

Now we can complete the
\begin{proof}[Proof of Theorem~\ref{T1.2}]
The existence of solution follows from Theorem~\ref{T3.5} and~\ref{T4.1} whereas \eqref{ET1.2B} follows from Theorem~\ref{T4.2}. 
Uniqueness is given by Theorem~\ref{T5.2} in combination with Theorem~\ref{T2.2}.
\end{proof}

Applying Lemma~\ref{L5.1} we have the following
\begin{thm}\label{T1.3}
Let $f_1, f_2$ be two continuous function
satisfying \hyperlink{F1}{(F1)}--\hyperlink{F2}{(F2)}. Let $\lambda^*(f_i)$
denote the critical value in \eqref{lam-crit} corresponding to $f_i$, $i=1,2$.
Then for $f_1\lneq f_2$ we have $\lambda^*(f_1)<\lambda^*(f_2)$.
\end{thm}

\begin{proof}
Let $\lambda_i^*=\lambda^*(f_i)$ for $i=1,2$. From the definition \eqref{lam-crit} it is evident that $\lambda^*_1\leq \lambda^*_2$. Suppose, on the contrary, that $\lambda^*_1= \lambda^*_2$.
Let $u_i$ be the non-negative solution corresponding to $\lambda^*_i$, that is, 
\begin{equation}\label{ET1.3A}
\cL u_i - f_i +\lambda^*_i=0\quad \text{in}\; \rd,
\end{equation}
for $i=1,2$. Since $f_1\leq f_2$, it then follows that
$$
\cL u_2- f_1 + \lambda^*_1\geq \cL u_2- f_2 + \lambda^*_2=0\quad \text{in}\; \rd.
$$

By Theorem~\ref{T2.2} we also have $u_1, u_2\in C^1(\rd)$.
From the proof of Theorem~\ref{T5.2}, we then have $u_1=u_2 +c$ for some constant $c$. Plugging this information in \eqref{ET1.3A} we get $f_1=f_2$ in $\rd$ which contradicts the hypothesis
$f_1\lneq f_2$. Thus we must have $\lambda^*_1<\lambda^*_2$. Hence the proof.
\end{proof}

\section{Stochastic characterization of $\lambda^*$}
\label{secstoch}

The main goal of this section is to prove Theorem~\ref{T1.4} which is
a generalized version of Theorem~\ref{TeoIntro2} mentioned in the introduction. In this section, we fix $-I=(-\Delta)^s$, the fractional Laplacian.
More precisely,
consider the problem
\begin{equation}\label{E6}
(-\Delta)^s u + H(x, \grad u) = f - \lambda \quad \mbox{in}\; \rd,
\end{equation}
where $H$ satisfies \hyperlink{H1'}{(H1')}, \hyperlink{H2'}{(H2')} and \hyperlink{H3}{(H3)}.
We also assume $f$ to be locally H\"{o}lder continuous. By
Theorem~\ref{T1.2} we see that $u\in C^{2s+}(\rd)$.
We next characterize $\lambda^*$ using the underlying 
stochastic control problem. To do so, we need the following additional assumption

\medskip

\begin{itemize}
\item[\hypertarget{H4}{(H4)}] For each $x \in \rd$, $p\mapsto H(x,p)$ is strictly convex and 
continuously differentiable. Also, the exists $\ell\in C(\rd\times\rd)$,
the Lagrangian,  which is
bounded from below, strictly convex in the second variable and satisfies
$$H(x, p)=\sup_{\xi \in \rd}\{p\cdot\xi -\ell(x, \xi)\}, \quad x, p \in\rd.$$
\end{itemize}

Notice that $H(x,p) = \frac{1}{m}|p|^m$ satisfies the above assumption with $\ell(x,\xi) = \frac{1}{m'}|\xi|^{m'}$, and $m'$ the H\"older conjugate of $m$.

Now, consider any non negative, classical solution $u$ to
\begin{equation}\label{E1.7}
\cL u - f + \lambda^*= (-\Delta)^s u +H(x, \grad u) - f+ \lambda^*=0\quad \text{in}\; \rd.
\end{equation} 

Since $H$ is the Fenchel--Legendre transformation of $\ell$, it is well known that
$$H(x, p)=\xi\cdot p -\ell(x, \xi)\quad \text{where}\; \xi =\grad_p H(x, p).
$$

If $u$ is a classical  solution to \eqref{E1.7}, 
$\grad u$ is continuous. Letting $b_u(x)=\grad_p H(x, \grad u(x))$, we observe that
$$
H(x, \grad u)= b_u(x)\cdot \grad u(x)-\ell(x, b_u(x))
=\sup_{\xi\in\rd} \{\grad u\cdot \xi - \ell(x, \xi)\}.$$

Let us define
\begin{equation}\label{G2}
\cG(x, \xi)=f(x)+\ell(x, \xi),
\end{equation}
and the operator
$\sA_u$ as follows
$$\sA_u \varphi(x)=-(-\Delta)^s \varphi(x) - b_u(x)\cdot \grad \varphi(x),\quad \varphi\in C^{2s+}(\rd)\cap L^1(\omega_s).$$

It is then easily seen from \eqref{E1.7} that
\begin{equation}\label{E1.8}
\sA_u u(x) + \cG(x, b_u(x))=-(-\Delta)^s u + \inf_{\xi\in\rd}\{-\xi\cdot \grad u + \cG(x, \xi)\}=\lambda^*.
\end{equation}


Thus, we can regard the problem in the context of stochastic ergodic control problem as it is explained in the introduction, this time with running cost given by $\mathcal G$ given in~\eqref{G2}. 

To make this section self-contained, we repeat some of the definitions provided in the introduction.
Let $\Omega:=\mathbb{D}([0, \infty):\rd)$ be the space of all right
continuous $\rd$ valued functions on $[0, \infty)$ with finite left limit, and
let $X:\Omega\to \rd $ denote the canonical coordinate process, that is,
$$X_t(\omega)=\omega(t)\quad \text{for all}\; \omega\in \Omega.$$
By $\{\mathfrak{F}_t\}$ we denote the filtration of $\sigma$-algebras generated by $X$.
Given a set $D$, by $\breve\uptau_D$ we denote the first return time to 
$D$, that is,
$$\breve\uptau_D=\inf\{t>0\; :\; X_t\notin D^c\}.$$

\begin{defi}\label{D1.2}
Let $\Usm$ be the set of all stationary Markov controls $\zeta$ satisfying the following:
\begin{itemize}
\item[(i)] The martingale problem for $(\sA_\zeta, C^\infty_c(\rd))$ with initial point $x\in\rd$ is well-posed for all $x\in\rd$ (see Definition~\ref{D1.1}).
\item[(ii)] For any non-empty compact set $\mathscr{B}$,
containing $\{x\in \rd\; :\; \cG(x, \xi)-\lambda^*\leq 1\; \text{for some}\; \xi\}$, we have $\Prob^\zeta_x(\breve\uptau_{\mathscr{B}}<\infty)=1$ for all $x\in {\mathscr{B}}^c$, 
where $\Prob^{\zeta}_x$ denotes the unique probability measure solving the martingale problem for $(\sA_\zeta, C^\infty_c(\rd))$ with initial point $x$.
\end{itemize}
\end{defi}

 By Lemmas~\ref{L6.3} and ~\ref{L6.4} it follows that $\Usm$ is non-empty.
Now we can state our main result of this section.
\begin{thm}\label{T1.4}
Let \hyperlink{H1'}{(H1')}, \hyperlink{H2'}{(H2')}, \hyperlink{H3}{(H3)}, \hyperlink{H4}{(H4)}; \hyperlink{F1}{(F1)}-\hyperlink{F2}{(F2)} hold and $f$ is locally H\"{o}lder continuous. Then for any $\zeta\in \Usm$ and compact set
$\mathscr{B}\supset \{x\in \rd\; :\; \cG(x, \xi)-\lambda^*\leq 1\; \text{for some}\; \xi\}$, we have
\begin{equation}\label{ET1.4A}
u(x)\leq \Exp^\zeta_x\left[\int_0^{\breve\uptau_{\mathscr{B}}}(\cG(X_s, \zeta(X_s))-\lambda^*)\ds\right] 
+ \Exp_x^\zeta[u(X_{\breve\uptau_{\mathscr{B}}})]\quad \text{for}\; x\in \mathscr{B}^c,
\end{equation}
where $u$ is the solution to~\eqref{E6} with $\lambda = \lambda^*$ given by Theorem~\ref{T1.2}, and 
$\Exp^\zeta_x[\cdot]$ denotes the expectation operator with respect to the probability measure $\Prob^{\zeta}_x$.
On the other hand, if there exists a nonnegative, classical
solution $v$ to~\eqref{E6} for some $\lambda\in\R$, and $(v, \lambda)$ satisfies \eqref{ET1.4A}, then we have $\lambda=\lambda^*$ and $v=u +c$ for some constant $c$.
\end{thm}

In the remaining part of the section we prove Theorem~\ref{T1.4}. For simplicity, we assume all the hypotheses of the theorem hold.

First, we prove \eqref{ET1.4A} in Lemmas~\ref{L6.1} and~\ref{L6.2} below. Recall the set of control $\Usm$ from
Definition~\ref{D1.2}.

\begin{lem}\label{L6.1}
Consider $\zeta\in \Usm$ and a compact set $\mathscr{B}$
satisfying the conditions of Definition~\ref{D1.2}(ii). 
Let $w_\alpha$ be the solution we obtain in Theorem~\ref{T3.2}. Then
we have
$$ w_\alpha(x)\leq \Exp^\zeta_x\left[\int_0^{\breve\uptau} 
e^{-\alpha s} \cG(X_s, \zeta(X_s)) \ds\right] + 
\Exp_x[e^{-\alpha \breve\uptau} w_\alpha(X_{\breve\uptau})]$$
for all $x\in \mathscr{B}^c$ satisfying $\Exp_x^\zeta[\breve\uptau]<\infty$, where
$$\breve\uptau=\inf\{t> 0\; :\; X_t\in \mathscr{B}\}.$$
\end{lem}

\begin{proof}
Recall the function $W_{n_k}$ that locally converges to $w_\alpha$ and
\begin{equation*}
\begin{split}
\cL W_{n_k}-f+\alpha W_{n_k}&=0 \hspace{3mm} \text{in}\hspace{2mm} B_{n_k},
\\
W_{n_k}&=0 \hspace{3mm} \text{in}\hspace{2mm} B^c_{n_k}.
\end{split}
\end{equation*}
Since $f$ is locally H\"{o}lder continuous, we have 
$W_{n_k}\in C^{2s+}(B_{n_k})$ by Theorem~\ref{T2.2}.
Choose $n_k$ large enough so that 
$\mathscr{B}\Subset B_{n_k}$. Set $D_k= B_{n_k}\setminus \mathscr{B}$. Let $\tilde{W}$ be the unique solution to
$$\sA_\zeta \tilde{W} + \cG-\alpha \tilde{W}=0\quad \text{in}\; D_k,
\quad \text{and} \quad \tilde{W}=0\quad \text{in}\; B^c_{n_k},
\quad \tilde{W}=W_{n_k}\; \text{in}\; \mathscr{B}.$$
Existence of viscosity solution follows from \cite[Corollary~5.7]{MOU-2017}
whereas regularity and uniqueness are established in \cite[Theorem~2.3]{BK22}. It is also shown in \cite[Theorem~2.3]{BK22} that 
$\tilde{W}\in C^{2s+}(D_k)\cap C_b(\rd)$. Now, note that
$$\cL \tilde{W} - f + \alpha \tilde{W}
\geq -\sA_\zeta \tilde{W} - \cG+\alpha \tilde{W}=0.
$$
As both $\tilde{W}$ and $W_{n_k}$ are classical solutions in $D_k$, we can apply
the comparison principle we obtain that $W_{n_k}\leq \tilde{W}$ in $\rd$. { Indeed, since $H$ is continuously differential in the second variable,
we note that
$$ (-\Delta)^s (\tilde{W}-W_{n_k}) + h(x)\cdot (\grad \tilde{W}-\grad W_{n_k}) + \alpha (\tilde{W}-W_{n_k})\geq 0 \quad \text{in}\; D_k,
$$
for some continuous function $h$ in $D_k$, and $(\tilde{W}-W_{n_k})\geq 0$ in $D^c_k$. Thus we can apply comparison principle.}
Using \cite[Lemma~4.3.2]{EK86} and Lemma~\ref{L5.1}, we see that 
$$\tilde{M}^\psi_t:= e^{-\alpha t}\psi(X_t)-\psi(X_0)
-\int_0^t e^{-\alpha s}(\sA_v \psi(X_s)-\alpha\psi(X_s))ds,$$
is $\mathfrak{F}_t$-martingale, for $\psi\in C^\infty_c(\rd)$. 
Let $\uptau_n$ be the first exit time from the ball $B_n(0)$. Set $n$ large enough
so that $\mathscr{B}\Subset B_n(0)$. Define
$\upsigma_n=\Breve\uptau\wedge \uptau_n$. Then $\upsigma_n$ denotes the 
first exit time from the annulus $B_n(0)\cap\mathscr{B}^c$, that is,
$$
\upsigma_n=\inf\{t>0\; :\; X_t\notin B_n(0)\cap\mathscr{B}^c\}.
$$
By optional sampling theorem \cite[Theorem~2.2.13]{EK86} it then follows that
$$
M^\psi_{t\wedge \upsigma_n}=e^{-\alpha (t\wedge \upsigma_n)}\psi(X_{t\wedge \upsigma_n})-\psi(x)-
\int_0^{t\wedge\upsigma_n} e^{-\alpha s}(\sA_\zeta\psi(X_s)-\alpha\psi(X_s))\ds
$$
also forms a martingale under $\Prob^\zeta_x$ with respect to the stopped filtration $\mathfrak{F}_{t\wedge\upsigma_n}$. Hence
\begin{equation*}
\psi(x)=
\Exp^\zeta_x[e^{-\alpha (t\wedge \upsigma_n)}\psi(X_{t\wedge \upsigma_n})] + 
\Exp^\zeta\left[\int_0^{t\wedge\upsigma_n}e^{-\alpha s}(\sA_\zeta\psi(X_s)-\alpha\psi(X_s))\ds\right]
\end{equation*}
for all $\psi\in C^\infty_c(\rd)$ and $t> 0$. By a standard approximation argument, it is easily seen that
\begin{equation}\label{EL6.1A}
\psi(x)=
\Exp^\zeta_x[e^{-\alpha (t\wedge \upsigma_n)}\psi(X_{t\wedge \upsigma_n})] + 
\Exp^\zeta\left[\int_0^{t\wedge\upsigma_n}e^{-\alpha s}(\sA_\zeta\psi(X_s)-\alpha\psi(X_s))\ds\right]
\end{equation}
for all $\psi\in C_b(\rd)\cap C^{2s+}(\overline{B_n(0)\cap \mathscr{B}})$. Now, for $n<n_k$, if we take $\psi=\tilde{W}$ and
then let $t\to\infty$, using the dominated convergence theorem we obtain
\begin{equation*}
\tilde{W}(x)= 
\Exp^\zeta_x[e^{-\alpha  \upsigma_n}\tilde{W}(X_{\upsigma_n})] + 
\Exp^\zeta\left[\int_0^{\upsigma_n}e^{-\alpha s}\cG(X_s, \zeta(X_s))\ds\right].
\end{equation*}
Again by \cite[Lemma~3.7]{ABC16}, $\uptau_{B_n(0)}\to \uptau_{n_k}:=\uptau_{B_{n_k}(0)}$ as $n\to n_k$. Since $\tilde{W}=0$ in $B^c_{n_k}(0))$,
we obtain from above that 
\begin{align*}
\tilde{W}(x)&=\Exp^\zeta_x\left[\int_0^{\breve\uptau\wedge \uptau_{n_k}} 
e^{-\alpha s} \cG(X_s, \zeta(X_s)) \ds\right] + 
\Exp^\zeta_x[e^{-\alpha \breve\uptau} 
W_\alpha(X_{\breve\uptau})1_{\{\breve\uptau<\uptau_{n_k}\}}]
\\
&\leq 
\Exp^\zeta_x\left[\int_0^{\breve\uptau} 
e^{-\alpha s} \cG(X_s, \zeta(X_s)) \ds\right] + 
\Exp^\zeta_x[e^{-\alpha \breve\uptau} 
W_{n_k}(X_{\breve\uptau})1_{\{\breve\uptau<\uptau_{n_k}\}}]
\end{align*}
for $x\in D_k$. Since $\cG$ is bounded from below and $\Exp_x^\zeta[\breve\uptau]<\infty$,
We obtain the desired result by letting $n_k\to\infty$.
\end{proof}

Now to prove \eqref{ET1.4A}, we only need to consider $\zeta\in\Usm$ satisfying $\Exp^\zeta_x[\breve\uptau]<\infty$, as
$\cG-\lambda^*>1$ in $\mathscr{B}^c$.
Recall from Theorem~\ref{T3.5} that $\bar{w}_{\alpha_{n_k}}\to u$, as 
$\alpha_{n_k}\to 0$ where $\bar{w}_\alpha(x):=w_\alpha(x)-w_\alpha(0)$.
From Lemma~\ref{L6.1} we see that
\begin{align*}
\bar{w}_\alpha(x)& \leq \Exp^\zeta_x\left[\int_0^{\breve\uptau} 
e^{-\alpha s} (\cG(X_s, \zeta(X_s))-\lambda^*) \ds\right]
+ (\lambda^*-\alpha w_\alpha(0)) \alpha^{-1}\Exp_x[1-e^{-\alpha\breve\uptau}] 
\\
&\quad + \Exp^\zeta_x[e^{-\alpha\breve\uptau} \bar{w}_\alpha(X_{\breve\uptau})]
\\
& \leq \Exp^\zeta_x\left[\int_0^{\breve\uptau} 
(\cG(X_s, \zeta(X_s))-\lambda^*) \ds\right]
+ |\lambda^*-\alpha w_\alpha(0)|\Exp_x[\breve\uptau] 
 + \Exp^\zeta_x[e^{-\alpha\breve\uptau} \bar{w}_\alpha(X_{\breve\uptau})],
\end{align*}
using the inequality $1-e^{-x}\leq x$ for all $x\geq 0$. 
Letting $\alpha_{n_k}\to 0$ to obtain the following result.

\begin{lem}\label{L6.2}
Let $u$ be the solution obtained by Theorem~\ref{T3.5}. Then it holds that
\begin{equation*}
u(x)\leq \Exp^\zeta_x\left[\int_0^{\breve\uptau}(\cG(X_s, Dv(X_s))-\lambda^*)\ds\right] + \Exp^\zeta_x[u(X_{\breve\uptau})],
\end{equation*}
for all $x\in \mathscr{B}^c$ and $\zeta\in\Usm$.
\end{lem}

Now we consider the second part of Theorem~\ref{T1.4}. Consider a pair $(v, \lambda)$ with $v\geq 0$, satisfying
$$\cL v -f +\lambda=0\quad \text{in}\; \rd.$$
From \eqref{lam-crit} it follows that $\lambda\geq \lambda^*$ and by Theorem~\ref{T2.2} we also have $v\in C^{2s+}(\rd)$.
Setting $b_v(x)=\grad_p H(x, \grad v(x))$ we observe from \eqref{E1.8} that
\begin{equation}\label{E6.2}
\sA_v v(x) + \cG(x, b_v(x))=\lambda\quad \text{in}\; \rd.
\end{equation}
In Lemma~\ref{L6.3} and ~\ref{L6.4} we show that $b_v\in\Usm$.
Let us first show that the martingale problem for $(\sA_v, C^\infty_c(\rd))$ with initial value $x$ is indeed well-posed.
To this end, we consider the perturbed operator
$$\sA_n \varphi(x)=-(-\Delta)^s \varphi(x) - 1_{B_n(0)}b_v(x)\cdot \grad \varphi(x) \quad \varphi\in C_c^\infty(\rd).$$
Note that $\sA_n=\sA_v$ on $C^\infty_c(B_n(0))$. Since 
$1_{B_n(0)}b_v$ is bounded, and therefore, lies in the Kato class
$\mathbb{K}_{d, 2s-1}$, by \cite[Theorem~1.2]{CW16} the martingale problem corresponding to
$(\sA_n, C^\infty_c(\rd))$ is well-posed (see also, \cite{CKS-12}).
Thus, there exists a unique probability measure $\Prob^n$ on 
$\Omega$ that corresponds to the martingale problem $(\sA_n, C^\infty_c(\rd))$ with initial point $x$.

\begin{lem}\label{L6.3}
For every $x\in\rd$, the martingale problem for $(\sA_v, C^\infty_c(\rd))$ with initial point $x\in\rd$ is well-posed.
\end{lem}

\begin{proof}
By \cite[Theorem~4.6.1]{EK86} we know that the stopped martingale problem for $\sA_n$ in any open set $U\subset\rd$ is well-posed.
Thus, in view of \cite[Theorem~4.6.3]{EK86} and  it is enough
to show that
\begin{equation}\label{EL7.1A}
\lim_{n\to\infty}\Prob^n(\sigma_n\leq t)=0\quad \text{for all}\; t>0,
\end{equation}
where
$$\sigma_n=\inf\{t\; : X_t\notin B_n(0)\; \text{or}\; X_{t-}\notin B_n(0)\}.$$
For a bounded domain $D$,
by $\uptau_D$ we denote the first exit time from $D$, that is,
$$\uptau_D=\inf\{t>0 \; : X_t\notin D\}.$$
We note that $\{X_{t\wedge \uptau_D}\; :\; t\geq 0\}$ has the same law under $\Prob_n$ whenever $D\subset B_n(0)$. This follows
from the well-posedness of the martingale problem for $\sA_n$ \cite[Theorem~4.6.1]{EK86}.
 Let 
$D\subset B_n(0)$ and $\psi\in C_c^{\infty}(\rd)$. Then 
$$\psi(X_t)-\psi(x)-\int_0^t \sA_n\psi(X_s)\ds$$
forms a martingale under $\Prob^n$. By optional sampling theorem \cite[Theorem~2.2.13]{EK86} it then follows that
$$\psi(X_{t\wedge \uptau_D})-\psi(x)-\int_0^{t\wedge\uptau_D} \sA_n\psi(X_s)\ds$$
also forms a martingale under $\Prob^n$. Since $\sA_n=\sA_v$ in $D$, it follows that
$$\psi(X_{t\wedge \uptau_D})-\psi(x)-\int_0^{t\wedge\uptau_D} \sA_v\psi(X_s)\ds$$
forms a martingale under $\Prob^n$ with respect to the stopped filtration $\mathfrak{F}_{t\wedge\uptau_D}$. Hence
\begin{equation*}
\Exp^n[\psi(X_{t\wedge \uptau_D})]=
\psi(x) + \Exp^n\left[\int_0^{t\wedge\uptau_D}\sA_v\psi(X_s)\right]
\end{equation*}
for all $\psi\in C^\infty_c(\rd))$, where $\Exp^n[\cdot]$ denotes the expectation operator with respect to the 
probability measure $\Prob^n$. By a standard approximation argument, it is easily seen that
\begin{equation}\label{EL7.1B}
\Exp^n[\psi(X_{t\wedge \uptau_D})]=
\psi(x) + \Exp^n\left[\int_0^{t\wedge\uptau_D}\sA_v\psi(X_s)\right]
\end{equation}
for all $\psi\in C_b(\rd)\cap C^{2s+}(\bar{D})$. Recall from \eqref{E6.2} that
$$\sA_v v(y) =\lambda-\cG(y, \grad v(y)).$$
Since $\cG$ is coercive, we can find a compact set $K$ and a constant
$\kappa>0$ so that
$$\sA_v v(y)\leq \kappa 1_K(y)\quad \text{in}\; \rd.$$
Since $v$ is non-negative, it is also coercive by Lemma~\ref{L3.3}. Let 
$\varphi_k:[0, \infty)$ be a bounded, increasing $C^2$ concave function satisfying
$$\varphi_k(s)=s\quad \text{for}\; s\in [0, k],
\quad \varphi(s)=k+1\quad \text{for}\; s\in [k+1, \infty),$$
and $0\leq \varphi^\prime_k\leq 1$.
Then 
$$\sA_v\varphi_k(v)=-(-\Delta)^s \varphi_k(v) - b_v\cdot \grad\varphi_k(v)
\leq \varphi^\prime_k(v)(-(-\Delta)^s v - b_v\cdot Dv)
\leq \kappa 1_K(y).
$$
Choose $k$ large enough so that $\varphi_k(v(x))=v(x)$. Now applying 
\eqref{EL7.1B} we obtain
\begin{equation*}
\Exp^n[\varphi_k(v(X_{t\wedge\uptau_D}))]\leq v(x) + \kappa t,
\end{equation*}
for all $D\subset B_n(0)$. Letting $k\to\infty$ and applying Fatou's lemma we then have
\begin{equation}\label{EL7.1C}
\Exp^n[v(X_{t\wedge\uptau_D})]\leq v(x) + \kappa t,
\end{equation}
for all $t>0$. Since $X$ is right continuous with finite left limits, 
it follows that 
$$\{\sigma_n\leq t \}\subset \{\uptau_{B_{n-1}}\leq t\}.$$
Using the coercivity of $v$ and taking $D=B_{n-1}$ in \eqref{EL7.1C} we obtain
\begin{align*}
\Prob^n(\sigma_n\leq t)\leq \Prob^n(\uptau_{B_{n-1}}\leq t)
\leq [\inf_{B^c_{n-1}} v]^{-1} (\kappa t + v(x))
\to 0, \quad \text{as}\; n\to\infty.
\end{align*}
This gives us \eqref{EL7.1A}, completing the proof.
\end{proof}
Thus, by Lemma~\ref{L6.3},  there exists a unique probability measure $\Prob^v_x$ satisfying the martingale problem for $(\sA_v, C^\infty_c(\rd))$ with initial point $x$. We denote the 
corresponding expectation operator by $\Exp^v_x[\cdot]$.
Now we fix a  ball $\mathscr{B}_1$ so that $\cG(x, \xi)-\lambda>1$ for
$x\in\mathscr{B}^c_1$. Let $\breve\uptau_1$ be the first hitting time to
$\mathscr{B}_1$, that is,
$$\breve\uptau=\inf\{t> 0\; :\; X_t\in \mathscr{B}_1\}.$$
Below in Lemma~\ref{L6.4} we show that $b_v\in\Usm$.

\begin{lem}\label{L6.4}
Let $v\in C^{2s+}(\rd)\cap L^1(\omega_s)$ be a nonnegative solution to~\eqref{E6}.

Define $b_v(y)=\grad_p H(y, \grad v(y))$ and $\sA_v$ is defined as above. Then
we have $\Exp^v_x[\breve\uptau_1]<\infty$ for all 
$x\in \mathscr{B}_1^c$ and
\begin{equation}\label{EL5.2B}
v(x)\geq \Exp^v_x\left[\int_0^{\breve\uptau}(\cG(X_s, b_v(X_s))-\lambda^*)\ds\right] + \Exp^v_x[v(X_{\breve\uptau})],
\end{equation}
for all $x\in \mathscr{B}^c_1$. Furthermore, $b_v\in\Usm$.
\end{lem}

\begin{proof}
From Lemma~\ref{L6.3}  we already know that the martingale problem for $(\sA_v, C^\infty_c(\rd))$ is well-posed for any initial condition $x\in\rd$.
Let $\uptau_n$ be the first exit time from the ball $B_n(0)$. Define
$\upsigma_n=\Breve\uptau\wedge \tau_n$. Then $\upsigma_n$ denotes the 
first exit time from the annulus $B_n(0)\cap\mathscr{B}^c_1$, that is,
$$
\upsigma_n=\inf\{t>0\; :\; X_t\notin B_n(0)\cap\mathscr{B}^c_1\}.
$$
The arguments in Lemma~\ref{L6.1} then gives (see \eqref{EL6.1A})
\begin{equation}\label{EL5.2C}
\Exp_x[\psi(X_{t\wedge \upsigma_n})]=
\psi(x) + \Exp_x\left[\int_0^{t\wedge\upsigma_n}\sA_v\psi(X_s)\right]
\end{equation}
for all $\psi\in C_b(\rd)\cap C^{2s+}(\overline{B_n(0)\setminus\mathscr{B}})$. Consider the class of concave functions $\varphi_k$ from Lemma~\ref{L6.3} and notice that

$$\sA_v \varphi_k(v)\leq \varphi^\prime_k(v)(\lambda-\cG)\quad
\text{in}\; \rd,$$
for all $k$. Applying \eqref{EL5.2C} we thus obtain
\begin{align*}
v(x)=\varphi_k(v(x))\geq \Exp_x[\varphi_k(v(X_{t\wedge \upsigma_n}))]
+ \Exp_x\left[\int_0^{t\wedge\upsigma_k}\varphi^\prime_k(v)(\cG(X_s, b_v(X_s))-\lambda)\right].
\end{align*}
Let $k\to\infty$ and apply Fatou's lemma to get
\begin{align*}
v(x)\geq \Exp_x[v(X_{t\wedge \upsigma_n})]
+ \Exp_x\left[\int_0^{t\wedge\upsigma_n}(\cG(X_s, b_v(X_s))-\lambda)\right].
\end{align*}
Letting $t\to\infty$ and applying Fatou's lemma and monotone convergence theorem again we have
\begin{align*}
v(x)\geq \Exp_x[v(X_{\upsigma_n})]
+ \Exp_x\left[\int_0^{\upsigma_n}(\cG(X_s, b_v(X_s))-\lambda)\right]\quad x\in B_n(0)\setminus \mathscr{B}.
\end{align*}
Since $v$ is coercive, and $\upsigma_n=\Breve\uptau_1\wedge \uptau_n$,
letting $n\to\infty$ in the relation
$v(x)\geq \Exp_x[v(X_{\upsigma_n})]$ implies that 
$\Prob_x(\breve\uptau_1<\infty)=1$. Now we can let $n\to\infty$ and apply Fatou's lemma again to obtain
\begin{align*}
v(x)\geq \Exp^v_x[v(X_{\breve\uptau})]
+ \Exp^v_x\left[\int_0^{\breve\uptau}(\cG(X_s, b_v(X_s))-\lambda)\right].
\end{align*}
This gives \eqref{EL5.2B}. Since
$(\cG-\lambda)>1$ in $\mathscr{B}^c$, we also have 
$\Exp^v_x[\breve\uptau_1]<\infty$. Again, by \cite[Theorem~5.1]{ABC16}, $X$ is positive recurrent under $\Prob^v_x$, which in 
particular, implies that $\Prob_x^v(\breve\uptau_{\mathscr{B}}<\infty)=1$
for all $\mathscr{B}$ containing $\{x\in \rd\; :\; \cG(x, \xi)-\lambda^*\leq 1\; \text{for some}\; \xi\}$. Thus
$b_v\in\Usm$, completing the proof.
\end{proof}

Now we can complete the proof of Theorem~\ref{T1.4}.
\begin{proof}[Proof of Theorem~\ref{T1.4}]
\eqref{ET1.4A} follows from Lemma~\ref{L6.1} and ~\ref{L6.2}.
To prove the second part, consider 
a nonnegative, classical
solution $v$ to
$$\cL v - f + \lambda=0\quad \text{in}\; \rd,$$
so that $(v, \lambda)$ satisfies \eqref{ET1.4A}. In other words,
for any compact $\mathscr{B}$, as in Definition~\ref{D1.2}(ii), we have
\begin{equation*}
v(x)\leq \Exp^\zeta_x\left[\int_0^{\breve\uptau_{\mathscr{B}}}(\cG(X_s, \zeta(X_s))-\lambda)\ds\right] 
+ \Exp_x^\zeta[v(X_{\breve\uptau_{\mathscr{B}}})]\quad \text{for}\; x\in \mathscr{B}^c,
\end{equation*}
for all $\zeta\in\Usm$. Since $\lambda\geq \lambda^*$ by \eqref{lam-crit}, it follows from above that
\begin{equation}\label{ET1.4B}
v(x)\leq \Exp^\zeta_x\left[\int_0^{\breve\uptau_{\mathscr{B}}}(\cG(X_s, \zeta(X_s))-\lambda^*)\ds\right] 
+ \Exp_x^\zeta[v(X_{\breve\uptau_{\mathscr{B}}})]\quad \text{for}\; x\in \mathscr{B}^c.
\end{equation}
Now, consider the solution $(u, \lambda^*)$ obtained by Theorem~\ref{T3.5}. Let $b_u(x)=\grad_p H(x, \grad u(x))$. From 
Lemma~\ref{L6.4} we know that $b_u\in\Usm$. Fix a compact
set $\mathscr{B}$ containing $\{x\in \rd\; :\; \cG(x, \xi)-\lambda^*\leq 1\; \text{for some}\; \xi\}$. Using 
\eqref{EL5.2B} for the solution pair $(u, \lambda^*)$,
and \eqref{ET1.4B} and choosing
$\zeta=b_u$, we then have
\begin{align*}
v(x)&\leq \Exp^u_x\left[\int_0^{\breve\uptau_{\mathscr{B}}}(\cG(X_s, b_u(X_s))-\lambda^*)\ds\right] 
+ \Exp_x^u[v(X_{\breve\uptau_{\mathscr{B}}})],
\\
u(x)&\geq \Exp^u_x\left[\int_0^{\breve\uptau_{\mathscr{B}}}(\cG(X_s, b_u(X_s))-\lambda^*)\ds\right] 
+ \Exp_x^u[u(X_{\breve\uptau_{\mathscr{B}}})]
\end{align*}
for $x\in \mathscr{B}^c$. Now translate $v$ by adding a constant so that $\min_{\mathscr{B}}(u-v)=0$. Then, from the above representation, we have $u\geq v$ in $\rd$. Repeating the arguments of Theorem~\ref{T5.2} we see that $\lambda=\lambda^*$
and $u=v$ in $\rd$. This completes the proof.
\end{proof}

\subsection*{Acknowledgement.}
This research of Anup Biswas was supported in part by a SwarnaJayanti fellowship SB/SJF/2020-21/03.
Erwin Topp was partially supported by Fondecyt Grant no. 1201897.

\subsection*{Data availability} There is no data in the present work.

\subsection*{Conflicts of interest statement}
On behalf of all authors, the corresponding author states that there is no conflict of
interest.

\bibliographystyle{plain}
\bibliography{ref.bib}

\end{document}